\documentclass{article}

\usepackage{blindtext} 

\usepackage[sc]{mathpazo} 
\usepackage[T1]{fontenc} 
\usepackage{microtype} 

\usepackage[english]{babel} 

\usepackage[hmarginratio=1:1,top=32mm,columnsep=20pt]{geometry} 
\usepackage[hang, small,labelfont=bf,up,textfont=it,up]{caption} 
\usepackage{booktabs} 

\usepackage{lettrine} 

\usepackage{amsmath}
\usepackage{amssymb}
\usepackage{relsize}
\usepackage{graphicx}
\usepackage{amsthm}

\usepackage{enumitem} 
\setlist[itemize]{noitemsep} 

\usepackage{abstract} 

\usepackage{titlesec} 
\renewcommand\thesection{\Roman{section}} 
\renewcommand\thesubsection{\roman{subsection}} 
\titleformat{\section}[block]{\large\scshape\centering}{\thesection.}{1em}{} 
\titleformat{\subsection}[block]{\large}{\thesubsection.}{1em}{} 

\usepackage{fancyhdr} 
\usepackage{titling} 

\usepackage{hyperref} 

\pretitle{\begin{center}\Huge\bfseries} 
	\posttitle{\end{center}} 

\usepackage[utf8x]{inputenc}
\usepackage{amsmath}
\usepackage{graphicx}
\usepackage[colorinlistoftodos]{todonotes}
\usepackage{lipsum}
\usepackage{babel}
\usepackage{enumitem}
\usepackage{mathtools}
\usepackage{amsmath,amssymb,amsfonts,amsthm,bbm}

\newcommand{\h}{\mathfrak{H}}
\newcommand{\sm}{\left(\smallmatrix}
\newcommand{\esm}{\endsmallmatrix\right)}
\newcommand{\mat}{\begin{pmatrix}}
	\newcommand{\emat}{\end{pmatrix}}
\def\acts{\curvearrowright}
\newcommand\m[1]{\begin{pmatrix}#1\end{pmatrix}} 
\newtheorem{claim}{Claim}
\newtheorem{lemma}{Lemma}
\newtheorem{definition}{Definition}
\newtheorem{theorem}{Theorem}
\newtheorem{proposition}{Proposition}
\newtheorem{corollary}{Corollary}

\newtheorem{remark}{Remark}

\newtheorem{example}{Example}

     \def\CC{\mathbb{C}}

     \def\HH{\mathbb{H}}

     \def\MM{\mathbb{M}}
     \def\NN{\mathbb{N}}
     
     \def\PP{\mathbb{P}}
     \def\QQ{\mathbb{Q}}
     \def\RR{\mathbb{R}}

     \def\ZZ{\mathbb{Z}}

\def\calH{\mathcal H}

\title{On Traces of Singular Moduli} 
\author{%
	\textsc{Malik Amir} \\[1ex] 
	\normalsize McGill University \\ 
	\normalsize \href{mailto:malik.amir@mail.mcgill.ca}{malik.amir@mail.mcgill.ca} 
}
\date{\today} 
\renewcommand{\maketitlehookd}{%
	\begin{abstract}
		\noindent  
		This article is an overview of Zagier's and Kim's work on traces of singular moduli. We give more detailed or new proofs to some of their results and also describe some algorithms to compute spaces of Jacobi forms and weight $3/2$ modular forms. Also, for $p$ a prime dividing the order of the Monster group, we give an explicit construction of the Hauptmodul $j_p$ of $\Gamma_0(p)^*$ using Rademacher series and Rademacher sums.
	\end{abstract}
}

\begin{document}
	
	\maketitle
	
	\section*{Acknowledgement}
	I would like to express my deepest gratitude to Professor Henri Darmon and Professor Sarah Harrison who introduced me to the theory of modular forms. Their patience and encouragements have incessantly pushed me to surpass myself. I would like also to thank Mr. Kelisky and the \textit{Kelisky Science Undergraduate Research Award} without whom this project would not have been possible. 
	\section{Traces of Singular Moduli à la Zagier}
	We call \textit{Singular Moduli} the values assumed by the $j$-function at quadratic irrationalities of $\HH$, namely, points $\alpha\in \HH$ solutions to integer quadratic polynomials with negative discriminant. This gives us a natural way to study singular moduli via positive definite binary quadratic forms of discriminant $-d<0$
	$$Q_d=\{Q=[a,b,c]=aX^2+bXY+cY^2:-d=b^2-4ac\}$$ where we recall that positive definiteness means $a>0$. Consider $Q(X,1)\in Q_d$ and $\alpha_Q\in \HH$ to be its unique root lying in the upper half-plane. A first fact is that singular moduli are algebraic integers. Before proving this, we state and prove some results about matrices, polynomials and modular functions. We will usually reserve the letter $\Gamma$ for $SL_2(\ZZ)$ unless otherwise stated.
	\begin{proposition}
		$j(\tau)=\frac{E_4(\tau)^3}{\Delta(\tau)}\in M^!_0(\Gamma)$.
	\end{proposition}
	\begin{proof}
		The pole at infinity follows from the fact that $\Delta$ has a zero of order one there. Furthermore, $j$ transforms like a modular form of weight 0 under the modular group since $E_4^3$ and $\Delta$ are in $M_{12}(SL_2(\ZZ))$.
	\end{proof}
	There is a natural action of $\Gamma$ on quadratic forms $Q(X,Y)$ given by $$\Gamma\acts\ Q_d\,:\,\gamma\cdot Q(X,Y)=Q(AX+BY,CX+DY)$$
	\begin{proposition}
		The value of $j$ at $\alpha_Q$ depends only on the $SL_2(\ZZ)$ equivalence class of $\alpha_Q$ and not on the choice of representative.
	\end{proposition}
	\begin{proof}
		Simply note that $j\in M_0^!(SL_2(\ZZ))$ so $j(\gamma\tau)=j(\tau)\,\,\forall \gamma\in SL_2(\ZZ)$. This implies that $j(\alpha_Q)=j(\gamma\alpha_Q)$. 
	\end{proof}
	The above proposition implies that we can always pick a representative lying in the fundamental domain. Here are some singular modulus
	$$j(i)=1728,\,j(\frac{1+i\sqrt{3}}{2})=0,\,j(i\sqrt{2})=8000,\,j(\frac{1+i\sqrt{7}}{2})=-3375,\,j(\frac{1+i\sqrt{15}}{4})=\frac{-191025+85995\sqrt{5}}{2}$$ Consider $\tau\in \HH$. We say that $\tau$ has complex multiplication, or is a CM point, if there exists a matrix $M$ of determinant $det(M)=n>0$, not being scalar, such that $\tau=M\tau$. Note that this property is equivalent to asking that $\tau$ satisfies an equation of the form $a\tau^2+b\tau+c=0,\,a,b,c\in\ZZ$. In order to study this property, we consider the map $\tau\mapsto M\tau$ for fixed $\tau\in \HH$ and it induces a map $j(\tau)\to j(M\tau)$ which runs over a finite set. To see this, note that for $\gamma\in \Gamma$, we have $j(M\tau)=j(\gamma M\tau)$ and by showing that $\Gamma\backslash M_n$ is finite, we get $$\{j(M\tau):M\in M_n\}=\{j(M\tau):M\in \Gamma\backslash M_n\}$$
	\begin{proposition}
		$$\mathcal{M}_n^*\triangleq\Gamma\backslash \mathcal{M}_n=\{\m{a&b\\0&d}:ad=n,0\leq b<d\}$$ where $\mathcal{M}_n=\{M\in M_2(\ZZ):det(M)=n\}$ is a full set of representatives. Furthermore $$\Big|\Gamma\backslash \mathcal{M}_n\Big|=\sum_{d|n}d=\sigma_1(n)$$ This implies that $$\{j(M\tau):M\in M_n\}=\{j(\frac{a\tau+b}{d}):ad=n,0\leq b<d\}$$ and $\tau$ has complex mutliplication $(CM)$ if and only if $j(\tau)$ lies in this set.
	\end{proposition}
	\begin{proof}
		Let $W_m=\{\m{a&b\\0&d}:ad=m,0\leq b<d\}$. If $\xi=\m{a&b\\0&d}\in W_m$, let $\Lambda_\xi$ a sublattice of $\Lambda=<w_1,w_2>$ (with $\Im{(\frac{w_1}{w_2})}>0$) generated by $w_1'=aw_1+bw_2$ and $w_2'=dw_2$. The map $\xi\to \Lambda_\xi$ is a bijection of $W_m$ onto the set $\Lambda(m)$ of sublattices of index $m$ in $\Lambda$. First, the fact that $\Lambda_\xi$ belongs to $\Lambda(m)$ follows from the fact that $det(\xi)=ad=m$. Conversly, let $\Lambda'\in\Lambda(m)$ a lattice of index m. Define $Y_1=\Lambda/(\Lambda'+\ZZ w_2)$ and $Y_2=\ZZ w_2/(\Lambda'\cap\ZZ w_2)$. These are cyclic groups generated respectively by the images of $w_1$ and $w_2$. Let $a$ and $d$ be the order of each of these groups. The exact sequence $0\to Y_2\to \Lambda/\Lambda'\to Y_1\to 0$ shows that $ad=n$. If $w_2'=dw_2$, then $w'_2\in \Lambda'$. On the other hand, there exists $w_1'\in \Lambda'$ such that $w_1'\equiv aw_1[\ZZ w_2]$. It is clear that $w_1'$ and $w_2'$ form a basis of $\Lambda'$. Moreover, we can write $w_1'$ in the form $w_1'=aw_1+bw_2$ with $b\in \ZZ$ where $b$ is uniquely determined modulo $d$. If we impose on $b$ the condition $0\leq b<d$, this fixes $b$, thus also $w_1'$. Thus we have associated to every $\Lambda'\in \Lambda(m)$ a matrix $\xi(\Lambda')\in W_m$ and we are now left simply by showing that the maps $\xi\mapsto \Lambda_\xi$ and $\Lambda'\mapsto\xi(\Lambda')$ are inverses. The last claim follows by definition.
	\end{proof}
	\begin{example}
		If $p$ is a prime, the elements of $\mathcal{M}_p^*$ are the matrix $\m{p&0\\0&1}$ and the matrices $\m{1&b\\0&p}$ with $0\leq b<p$.
	\end{example}
	\begin{theorem}
		$\forall n\geq 1$, $\exists \Psi_n(X,Y)\in\ZZ[X,Y]$ called the $n^{th}$ modular polynomial satisfying $$\{\text{roots of }\Psi_n(X,j(\tau))\}=\{j(M\tau):M\in \Gamma\backslash \mathcal{M}_n\}$$ i.e. $$\Psi_n(X,j(\tau))=\prod_{M\in \Gamma\backslash \mathcal{M}_n}(X-j(M\tau))=\prod_{ad=n,\,0\leq b<d}(X-j(a\tau+b/d))$$ with $deg(\Psi_n(X,j(\tau)))=\sigma_1(n)$.
	\end{theorem}
	\begin{proof}
		Let $\mathcal{M}_n$ as above and recall that $\Gamma$ acts on $\mathcal{M}_n$ by left multiplication with finitely many orbits.
		\begin{claim}
			$\prod_{M\in \Gamma\backslash \mathcal{M}_n}(X-j(M\tau))$ for $\tau\in\HH,X\in \CC$ is a polynomial in $X,j(\tau)$ with complex coefficients.
		\end{claim}
		\begin{proof}
			It is a polynomial in $X$ of degree $\sigma_1(m)$ whose coefficients are holomorphic functions of $\tau$ with at most exponential growth at infinity. Recall that any $\Gamma$-invariant holomorphic function in $\HH$ with at most exponential growth at $\infty$ is a polynomial in $j(\tau)$.
			Hence, we can define $\Psi_n(X,j(\tau))=\prod_{M\in \Gamma\backslash \mathcal{M}_n}(X-j(M\tau))$ and it is indeed a polynomial $\Psi_n(X,j(\tau))\in \CC[X,j(\tau)]$.
		\end{proof}
		Now we are left with showing that the coefficients are integers and we will have constructed our $n^{th}$ modular polynomial. Lets write $$\Psi_n(X,j(\tau))=\prod_{M\in \Gamma\backslash \mathcal{M}_n}(X-j(M\tau))=\prod_{ad=n,d>0}\prod_{b=0}^{d-1}(x-j(a\tau+b/d))$$ which can be rewritten using the Fourier expansion of $j(\tau)$ as $$=\prod_{ad=n,d>0}\prod_{b[d]}\Big(X-\sum_{n=-1}^\infty c_n\zeta_d^{bn}q^{an/d}\Big)$$ where as usual $q^\alpha=e^{2\pi i\alpha\tau}$ for some $\alpha\in \QQ$ and $\zeta_d=e^{2\pi i/d}$ is a $d$-th root of unity. $\Psi_n$ belongs a priori to the ring of Laurent series in $q^{1/d}$ with coefficients in $\ZZ[\zeta_d]$, i.e $\Psi_n\in \ZZ[\zeta_d][X](q^{1/d})$. Note that by applying the Galois automorphism $\zeta_d\mapsto \zeta_d^r$ with $r\in (\ZZ/d\ZZ)^\times$, we can replace $b$ by $br$ which still runs over $\ZZ/d\ZZ$. Hence we have invariance of the coefficients under Galois automorphisms of $\QQ(\zeta_d)$ and we must have that the coefficients are in fact integers. One can also use the fact that $\prod_{b[d]}(X-j(a\tau+b/d))$ is invariant under the translation $\tau\mapsto \tau+1$. This implies that any $q^{n/d}$ with $d\nmid n$ vanishes (otherwise we would make appear an extra phase $\zeta^n_d)$. Thus, any product $\prod_{b[d]}(X-j(a\tau+b/d))=\prod_{b[d]}(X-\sum_{n=-1}^\infty c_n\zeta_d^{bn}q^{an/d})$ belongs to $\ZZ[X](q)$ and this gives $\Psi_n(X,j(\tau))\in \ZZ[X](q)$. However, recall that $\Psi_n$ is a polynomial in $X$ and $j(\tau)$ and that $j(\tau)$ has leading coefficient $1/q$. This gives $q=j^{-1}-744j^{-2}-...\in \ZZ(1/j)$ and so $\Psi_n\in \ZZ[X](q)=\ZZ[X](1/j)$ but $\Psi_n$ is a polynomial in $j$ and $X$ so $\Psi_n(X,j(\tau))\in \ZZ[X,j]$ and the theorem is proved.
	\end{proof}	
	\begin{proposition}
		For $m$ not a perfect square, the polynomial $\Psi_n(j(\tau),j(\tau))$ is, up to sign, a monic polynomial of degree $\sigma_1^+(n)\triangleq \sum_{d|n}max(d,m/d)$.
	\end{proposition}
	\begin{proof}
		We will use the following identity $$\prod_{b[d]}(x-\zeta_d^by)=x^d-y^d$$ This gives us that $$\Psi_n(j(\tau),j(\tau))=\prod_{ad=n}\prod_{b[d]}(j(\tau)-j(a\tau+b/d))$$ $$=\prod_{ad=n}\prod_{b[d]}(q^{-1}-\zeta^{-b}_dq^{-a/d}+O(1))=\prod_{ad=n}(q^{-d}-q^{-a}+\text{ lower order terms })$$ and because $m$ is not a perfect square we have that $q^{-d}-q^{-a}\neq 0$ which gives for leading term $$\pm q^{-\sigma_1^+(n)}$$ Finally, to see that it is monic in $j$, simply recall that $j-744= q^{-1}+O(q)$.
	\end{proof}
	\begin{theorem}
		Let $\tau$ be a CM point. Then $j(\tau)$ is algebraic. 
	\end{theorem}
	\begin{proof}
		Consider $\Psi_n(j(\tau),j(\tau))=\prod_{ad=n}\prod_{b[d]}(j(\tau)-j(a\tau+b/d))$. It is a polynomial with integer coefficients. By assumption, $\tau$ is a CM point so by \textit{proposition 3} we have that $j(\tau)$ lies in $\{j(\frac{a\tau+b}{d}):ad=n,0\leq b<d\}$ and it follows that $j(\tau)$ is a zero of $\Psi_n$ and by definition is thus algebraic. 
	\end{proof}
	Finally we have as an immediate consequence 
	\begin{corollary}
		Singular moduli are algebraic.
	\end{corollary}
	\begin{definition}
		We denote by $h(-d)$ the class number of $-d$, i.e. the number of equivalence classes of primitive quadratic forms of determinant $-d$  $$h(-d)=\Big|\Gamma\backslash\big\{ Q=[a,b,c]\,:\,det(Q)=-d,\,gcd(a,b,c)=1\big\}\Big|$$ and thus the number of singular modulus for a fixed discriminant $-d$.
	\end{definition}
	Each of these values is an algebraic integer as seen previously and is of degree $h(-d)$ and these values form a full set of conjugates so that their sum is always an integer. More precisely, we have 
	\begin{proposition}
		$\mathcal{P}_d(X)=\prod_{Q\in \Gamma\backslash Q_d}(X-j(\alpha_Q))\in\ZZ[X]$ and is irreducible. It follows that it is the minimal polynomial for singular moduli. Furthermore, each of the algebraic values $j(\alpha_Q)$ is of degree $h(-d)$ over $\QQ$ with conjugates $j(\alpha_{Q,i})$ for $1\leq i\leq h(-d)$.
	\end{proposition}
	Here are some of these values $$h(-3)=h(-4)=h(-7)=h(-8)=1, h(-15)=2$$ and the singular moduli are on \textit{page 1}. We can see the conjugate pair for $d=15$ when the class number is greater than 1.\\
	
	\begin{definition}
		We define the Hurwitz-Kronecker class number $H(d)$ as $$H(d)=\sum_{Q\in \Gamma\backslash Q_d}\frac{1}{w_Q}$$ where $w_Q$ is the size of the stabilizer subgroup of $\alpha_Q$, i.e. $w_Q=|Stab(\alpha_Q)|=2,3$ if $Q$ is $\Gamma$-equivalent to $[a,0,a]$ or $[a,a,a]$ respectively and $w_Q=1$ otherwise.\\
		For example, $H(3)=1/3,\,H(4)=1/2,\,H(15)=2$.
	\end{definition}
	With all these tools in hand, we are ready to state Borcherds theorem for scalar valued modualr forms \cite{Bo1}
	\begin{theorem}{(Borcherds Theorem)\\}
		Let $\xi(\tau)=-\frac{1}{12}+\sum_{1<d\equiv 0,3[4]}H(d)q^d=-\frac{1}{12}+\frac{q^3}{3}+...$ where $H(d)$ is the Hurwitz class number. For a weight $1/2$ weakly holomorphic modular form on $\Gamma_0(4)$ given by $$f(\tau)=\sum_{d\geq d_0\in \ZZ,\,d\equiv 0,1[4]}A(d)q^d$$ define $$\psi(f(\tau))=q^{-h}\prod_{d=1}^\infty (1-q^d)^{A(d^2)}$$ where $h$ is the constant term of $f(\tau)\xi(\tau)$. Then the map $\psi$ is an isomorphism between the additive group of weakly holomorphic modular forms on $\Gamma_0(4)$ with integer coefficients $A(d)$ satisfying $A(d)=0$ unless $d\equiv 0,1[4]$ and the multiplicative group of integer weight modular forms on $\Gamma$ with Heegner divisor, integer coefficients and leading coefficient 1. Under this isomorphism, the weight of $\psi(f(\tau))$ is $A(0)$ and the multiplicity of the zero of $\psi(f(\tau))$ at a Heegner point of discriminant $D<0$ is $\sum_{d>0}A(Dd^2)$.
	\end{theorem}
	
	\begin{example}
		The Eisenstein series $E_4(\tau)=1 + 240\sum_{n=1}^\infty \sigma_3(n)q^n$ has the product expansion $$E_4(\tau)=(1-q)^{-240}(1-q^2)^{26760}(1-q^3)^{-4096240}...=\prod_{n=1}^{\infty}(1-q^n)^{c(n)}$$ From the above isomorphism, we can conclude that there is a form in the Kohnen plus-space $M_{1/2}^{!,+}(\Gamma_0(4))$ with Fourier expansion $f(\tau)=\sum_n c(n)q^n$.
	\end{example}
	\begin{example}
		Consider $f(\tau)=12\theta(\tau)=12+24q+24q^4+24q^9+O(q^{16})\in M_{1/2}^{+}(\Gamma_0(4))$. By Borcherds theorem we get $\psi(f(\tau))=q\prod_{n=1}^\infty (1-q^n)^{24}=\Delta(\tau)$.
	\end{example}
	The above theorem has to be interpreted as a systematic way to lift or descend forms from one space to another and thus an interesting way to produce new modular forms in the desired space.
	\begin{definition}
		The trace of singular moduli of discriminant $-d$ is the sum $$\bold{t}(d)\triangleq \sum_{Q\in \Gamma\backslash Q_d}\frac{j(\alpha_Q)-744}{w_Q}$$ where $J(\tau)=j(\tau)-744$ is the canonical hauptmodul of $SL_2(\ZZ)$.
	\end{definition}
	Here are some of its values $$\bold{t}(3)=-248,\,\bold{t}(4)=492,\, \bold{t}(7)=-4119, \bold{t}(8)=7256$$
	It looks like the trace is supported only for integers $n\equiv 0,3[4]$ and by considering its generating series $\sum_d \bold{t}(d)$, we can expect that it is a modular form of half-integral weight on $\Gamma_0(4)$ and using the definition of the Kohnen plus-space, we can guess that the weight has to be $3/2$. In its amazing work, Zagier succeeded in relating the generating series of the trace to a weight $3/2$ weakly holomorphic modular form living in the Kohnen plus-space. Here is the weight $3/2$ modular form he constructed $$g(\tau)\triangleq \theta_1(\tau)\frac{E_4(4\tau)}{\eta(4\tau)^6}=q^{-1}-2+248q^3-492q^4+4119q^7+O(q^8)$$ where $\theta_1(\tau)=\sum_{n\in\ZZ}(-1)^nq^{n^2}$, $E_4(\tau)=1+240\sum_{n\geq 1}\frac{n^3q^n}{1-q^n}$ and $\eta(\tau)=q^{1/24}\prod_{n\geq 1}(1-q^n)$. Here the idea is that in order to obtain a modular form in the Kohnen plus-space with alternating terms, one has an easy candidate which is $\theta_1(\tau)$, the twist of the usual $\theta(\tau)$ function. The quotient here ensures that our form will start in $q^-1$. From here, it is possible to prove (and it will be done later) that such a form of weight $3/2$ is uniquely determined by its first and second non-zero Fourier coefficients, namely the leading term $q^{-1}$ and the constant term $-2$. By looking at the Fourier coefficients of $g(\tau)$ and the values of $\bold{t}(d)$, we see that they are the same up to a minus sign
	\begin{theorem}
		Write $g(\tau)=\sum_{d\geq -1}B(d)q^d$. Then $$\bold{t}(d)=-B(d),\,\forall d>0$$
		With this result, the task of fully describing the trace function is done. In 2008 for its PhD thesis, Paul M. Jenkins provided exact formulas for the trace using Maass-Poincaré series. He also derived nice p-adic properties of traces, congruences and a criteria for $p$-divisibility of class numbers of imaginary quadratic fields in terms of p-divisibility of traces of singular moduli.
	\end{theorem}
	There are two proofs of this result given in \cite{Za}. The first one that we will present now mostly in details is the longest and relies on recurence relations. The second one uses Borcherds theorem as a black box and a duality between weight $1/2$ and weight $3/2$ modular forms.
	\begin{proof}
		The idea is to find relations that the Fourier coefficients of $g(\tau)=\sum_n B(n)q^n$ must satisfy and see if the trace satisfies them too. These relations are recursions and determine uniquely the function whose Fourier coefficients satisfy them. They are given explicitely by $$B(4n-1)=240\sigma_3(n)-\sum_{2\leq r\leq \sqrt{4n+1}}r^2 B(4n-r^2),\,B(4n)=-2\sum_{1\leq r\leq \sqrt{4n+1}}B(4n-r^2)$$
		To study if the trace function satisfies these relations, we define the modified Hilbert class polynomial $\mathcal{H}_d$ $\forall d>0,\, d\equiv 0,3[4]$ as $$\mathcal{H}_{d}(X)=\prod_{Q\in \Gamma\backslash Q_d}(X-j(\alpha_Q))^{1/w_Q}$$ This function is $X^{1/3}$ times a polynomial in $X$ if $d/3$ is a square, $(X-1728)^{1/2}$ times a polynomial in $X$ if $d$ is a square and a polynomial in $X$ otherwise. After a little bit of algebra, we get $$\calH_d(j(\tau))=\prod_{Q\in\Gamma\backslash Q_d}(q^{-1}-J(\alpha_Q)+O(q))^{1/w_Q}=q^{-H(d)}(1-Tr(d)q+O(q^2))$$ Also, we define $$\Lambda_d(\tau)=\frac{-1}{2\pi i}\frac{d}{d\tau}log(\calH_d(j(\tau)))=\frac{-1}{2\pi i}\frac{dj(\tau)}{d\tau}\sum_{Q\in \Gamma\backslash Q_d} \frac{1}{w_Q}\frac{1}{j(\tau)-j(\alpha_Q)}$$ $$=\frac{E_4^2(\tau)E_6(\tau)}{\Delta(\tau)}\sum_{Q\in \Gamma\backslash Q_d} \frac{1}{w_Q}\frac{1}{j(\tau)-j(\alpha_Q)}=H(d)+Tr(d)q+O(q^2)$$ since $j=E_4^3/\Delta$ and thus $-j'(\tau)=-\frac{1}{2\pi i}\frac{dj(\tau)}{d\tau}=E^2_4E_6/\Delta$. Hence, $-2\pi i\Lambda_d(\tau)$ is a meromorphic modular form of weight $2$, holomorphic at infinity and having simple pole of residue $1/w_Q$ at each $\alpha \in\HH$ satisfying a quadratic equation over $\ZZ$ with discriminant $-d$. Such a form is uniquely characerized by these properties since there are no holomorphic modular forms of weight $2$ on $\Gamma$.\\
		\begin{proposition}
			$$\frac{E_4(\tau)E_6(\tau)}{\Delta(\tau)}\sum_{M\in \Gamma\backslash\mathcal{M}_n}\frac{(E_4|M)(\tau)}{j(\tau)-j(M\tau)}=\frac{1}{2}\sum_{r^2<4n}(n-r^2)\Lambda_{4n-r^2}(\tau)$$ where for $M\in \mathcal{M}_n$ we let $(E_4|M)=\frac{n^3}{(c\tau+d)^4}E_4(\frac{a\tau+b}{c\tau+d})$ so $E_4|(\gamma M)=E_4|M$.
		\end{proposition}
		\begin{proof}
			We note that both sides are meromorphic modular forms of weight 2 with simple poles and holomorphic at $\infty$. To show equality, it suffices to prove that they have the same residues. Let $M\alpha=\alpha$ and $\lambda=c\alpha+d$ where $c,d$ are the third and fourth entry of $M$. The residue of the LHS at $\tau=\alpha$ is $$\frac{E_4(\alpha)E_6(\alpha)}{\Delta(\alpha)}\frac{E_4|M(\alpha)}{j'(\alpha)-(n/\lambda^2)j'(M\alpha)}=\frac{E_4(\alpha)E_6(\alpha)}{\Delta(\alpha)}\frac{(n^3/\lambda^4)E_4(\alpha)}{(1-n/\lambda^2)j'(\alpha)}$$ $$=\frac{-1}{2\pi i}\frac{n^3/\lambda^4}{1-n/\lambda^2}=\frac{-1}{2\pi i}\frac{\overline{\lambda}^3}{\lambda-\overline{\lambda}}$$ since $n=\lambda\overline{\lambda}, M(\alpha,1)^t=\lambda(\alpha,1)^t,\lambda+\overline{\lambda}=r=\bold{t}(M)$. Therefore we obtain $$\frac{-1}{2\pi i}\sum_{r^2<4n}\sum_{\alpha\in \Gamma\backslash Q_{d=r^2-4n}}\frac{-\overline{\lambda}^3}{\lambda-\overline{\lambda}}=\frac{1}{4\pi i}\sum_{r^2<4n}\sum_{\alpha\in \Gamma\backslash Q_{4n-r^2}} \frac{\lambda^3-\overline{\lambda}^3}{\lambda-\overline{\lambda}}$$ $$=\frac{1}{4\pi i}\sum_{r^2<4n}\sum_{\alpha\in \Gamma\backslash Q_{4n-r^2}}  (\lambda^2+\lambda\overline{\lambda}+\overline{\lambda}^2)=\frac{1}{4\pi i}\sum_{r^2<4n}\sum_{\alpha\in \Gamma\backslash Q_{4n-r^2}}(\lambda+\overline{\lambda})^2-\lambda\overline{\lambda}$$ $$=\frac{1}{4\pi i}\sum_{r^2<4n}\sum_{\alpha\in \Gamma\backslash Q_{4n-r^2}}r^2-n$$ but this is exactly the sum of residues of the right hand side.
		\end{proof}
		Let $$S_n(\tau)\triangleq \frac{E_4(\tau)E_6(\tau)}{\Delta(\tau)}\sum_{M\in \Gamma\backslash\mathcal{M}_n}\frac{(E_4|M)(\tau)}{j(\tau)-j(M\tau)}$$ We write its $q$-expansion at infinity as $$S_n(\tau)=C_o+C_1q+O(q^2)$$ Note that we do not have negative powers of $q$ since our function is holomorphic at $\infty$. Since  $$\Lambda_{4n-r^2}(\tau)=H(4n-r^2)+\bold{t}(4n-r^2)q+O(q^2)$$ we have by the above proposition  $$S_n(\tau)=\frac{1}{2}\sum_{r^2<4n}(n-r^2)\Lambda_{4n-r^2}(\tau)=\frac{1}{2}\sum_{r^2<4n}(n-r^2)\big(H(4n-r^2)+\bold{t}(4n-r^2)q+O(q^2)\big)$$ $$=\frac{1}{2}\sum_{r^2<4n}(n-r^2)H(4n-r^2)+(n-r^2)\bold{t}(4n-r^2)q+O(q^2)$$ and finally $$C_0=\frac{1}{2}\sum_{r^2<4n}(n-r^2)H(4n-r^2)\text{ and }C_1=\frac{1}{2}\sum_{r^2<4n}(n-r^2)\bold{t}(4n-r^2)$$ Now we use another representation of $S_n(\tau)$, the one given by taking a set of representatives for $\Gamma\backslash \mathcal{M}_n$, namely $\mathcal{M}_n^*=\{\m{a&b\\0&d}:ad=n,0\leq b<d\}$. We obtain $$S_n(\tau)= \frac{E_4(\tau)E_6(\tau)}{\Delta(\tau)}\sum_{M\in \Gamma\backslash\mathcal{M}_n}\frac{(E_4|M)(\tau)}{j(\tau)-j(M\tau)}=\frac{E_4(\tau)E_6(\tau)}{\Delta(\tau)}\sum_{M\in \mathcal{M}_n^*}\frac{(E_4|M)(\tau)}{j(\tau)-j(M\tau)}$$ $$=\sum_{ad=n,d>0}a^3\Big(\frac{E_4(\tau)E_6(\tau)}{\Delta(\tau)}\frac{1}{d}\sum_{b[d]}\frac{E_4(\frac{a\tau+b}{d})}{j(\tau)-j(\frac{a\tau+b}{d})}\Big)$$ and if we define $$S_{a,d}(\tau)\triangleq \frac{E_4(\tau)E_6(\tau)}{\Delta(\tau)}\frac{1}{d}\sum_{b[d]}\frac{E_4(\frac{a\tau+b}{d})}{j(\tau)-j(\frac{a\tau+b}{d})}$$ we get $$S_n(\tau)=\sum_{ad=n,d>0}a^3S_{a,d}(\tau)$$
		Computing the $q$-expansion of $S_{a,d}$ we get  
		$$\frac{(1+240q+...)(1-504q+...)}{q(1-24q+...)}\frac{1}{d}\sum_{b[d]}\frac{1+240\sum_{l=1}^\infty\sigma_3(l)\zeta_d^{bl}q^{al/d}}{q^{-1}-\zeta_d^{-b}q^{-a/d}+O(q^>0)}$$ $$=(1-240q+O(q^2))\frac{1}{d}\sum_{b[d]}\frac{1+240\sum_{l=1}^\infty\sigma_3(l)\zeta_d^{bl}q^{al/d}}{1-\zeta_d^{-b}q^{1-a/d}+O(q^{>1})}$$ If $a<d$, we will transform $$\frac{1}{1-(\zeta_d^{-b}q^{1-a/d}+O(q^{>1}))}$$ by making use of geometric series, namely $$=(1-240q+O(q^2))\frac{1}{d}\sum_{b[d]}\frac{1+240\sum_{l=1}^\infty\sigma_3(l)\zeta_d^{bl}q^{al/d}}{1-\zeta_d^{-b}q^{1-a/d}+O(q^{>1})}$$ $$=(1-240q+O(q^2))\frac{1}{d}\sum_{b[d]}\Big ((1+240\sum_{l=1}^\infty\sigma_3(l)\zeta_d^{bl}q^{al/d})(\sum_{m=0}^\infty \zeta_d^{-bm}q^{m(1-a/d)}+O(q^{>1}))\Big)$$ $$=(1-240q+O(q^2))\frac{1}{d}\sum_{b[d]}\Big (\sum_{m=0}^\infty \zeta_d^{-bm}q^{m(1-a/d)}+O(q^{>1})\Big)$$ $$+\sum_{l=1}^\infty\sum_{m=0}^\infty\Big( 240\sigma_3(l)\zeta_d^{b(l-m)/d}q^{a(l-m)/d+m}+O(q^{>1}) \Big)$$ $$=(1-240q+O(q^2))\frac{1}{d}\sum_{b[d]}\Big(\sum_{l=1}^\infty \sum_{m=0}^\infty 240\sigma_3(l)\zeta_d^{b(l-m)/d}q^{a(l-m)/d+m}+O(q^{>1})\Big)$$ and we note that to avoid having nonintegral powers of $q$, we must have that $l\equiv m[d]$, otherwise we let the expression to be 0. This implies that $l-m$ is divisible by $d$ and that $(\zeta_d^{l-m})^{b/d}=1,\forall 0\leq b<d$. Hence, no term in the inner sums depend on $b$ and $\sum_{b[d]}1=d$ which now cancels the term $\frac{1}{d}$. We thus obtain $$=(1-240q+O(q^2))\sum_{l,m\geq 0,l\equiv m[d]}240\sigma_3(l)q^{a(l-m)/d+m}+O(q^2)$$ and finally, looking at the terms contributing to $q^0$ and $q^1$ we can rewrite the above as $$=1+(240\delta_{a,1}\sigma_3(n)+\delta_{a,d-1})q+O(q^2)$$ since the only pairs $(l,m)$ contributing are $(0,0),(1,1),(d,0)$ if $a=1$ and $(0,d)$ if $a=d-1$. \\
		For $a>d$, a similar calculation gives $$S_{a,d}(\tau)=\sum_{l,m\geq 0,l+m\equiv 0[d]}240\sigma_3(l)q^{-m+a(l-m)/d}(1-240q+O(q^2))$$ We can see that there is no constant term and thus we get $0+(-\delta_{a,d+1})q+O(q^2))$ since this time only the pair $(l,m)=(0,d)$ contributes. Summing we get $$C_0=\sum_{0<a<\sqrt{n},a|n}a^3\text{ and }C_1=240\sigma_3(n)-\left\{
		\begin{array}{ll}
		3n+1 & \text{if }4n+1 \text{ is a square} \\
		0, & \text{ otherwise }
		\end{array}
		\right.$$ with the last term coming from the factorization of $n$ as $ad$ with $a-d=\pm1$. This finishes the proof when $n$ is not a square.\\
		When $n$ is a square, everything goes exactly the same way except that we now have to compute the Fourier expansion of $S_{a,d}$ at $a=d$ too.\\
		In addition, by comparing the expressions obtained for $C_0$ we get the already well-known identity $$\sum_{r^2<4n}(n-r^2)H(4n-r^2)=\sum_{d|n}min(d,n/d)^3-\left\{
		\begin{array}{ll}
		n/2, & n \text{ is a square} \\
		0, & \text{ otherwise }
		\end{array}
		\right.$$ which together with $$\sum_{r^2<4n}H(4n-r^2)=\sum_{d|n}max(d,n/d)+\left\{
		\begin{array}{ll}
		1/6, & n \text{ is a square} \\
		0, & \text{ otherwise }
		\end{array}
		\right.$$ determines the Hurwitz-Kronecker class number $H(n)$ recursively.
	\end{proof}
	Now lets take a look at the proof involving Borcherds lift. As a consequence of Borcherds theorem, we have 
	\begin{theorem}
		For $d>0,d\equiv 0,3[4]$, $$\mathcal{H}_d(j(\tau))=q^{-H(d)}\prod_{n=1}^\infty(1-q^n)^{A(n^2,d)}$$ where $A(n^2,d)$ is the coefficient of a weight $1/2$ modular form $f_d\in M_{1/2}^{+,!}(\Gamma_0(4))$.
	\end{theorem}
	An immediate corollary of this result is 
	\begin{corollary}
		$\bold{t}(d)=-B(d)=A(1,d),\forall d>0$.
	\end{corollary}
	This is the first mysterious appearance of the coefficient duality between weight $3/2$ and $1/2$ modular forms on $\Gamma_0(4)$. We might be interested in computing a basis for each of these space and see if we can explicitely show the Zagier duality. To investigate this path, we recall some facts about the theory of half-integral weight modular forms.
	\begin{definition}
		The space of half-integral weight holomorphic modular forms on $\Gamma_0(4)$ is denoted $M_{k+1/2}(\Gamma_0(4))$ and consists of those modular forms holomorphic in $\HH$ and at the cusps and which transform under the action of $\Gamma_0(4)$ like $\theta^{2k+1}$ where $\theta(\tau)=\sum_{\ZZ}q^{n^2}$.
	\end{definition}
	\begin{definition}
		A very interesting subspace of $M_{k+1/2}(\Gamma_0(4))$ is the Kohnen plus-space $M_{k+1/2}^+$. It consists of modular forms $f=\sum_n a_nq^n$ whose Fourier coefficients satisfy $a_n\neq 0\iff (-1)^kn\equiv 0,1 [4]$.
	\end{definition}
	There is an isomorphism that was discovered by Kohnen in \cite{Koh} that explicitly gives the restriction to finding modular forms in $M^+_{k+1/2}$
	\begin{theorem}
		$$M^+_{k+1/2}(\Gamma_0(4))\cong M_{2k}(\Gamma)$$ and it preserves the space of cusp forms. We thus obtain that $M^+_{k+1/2}(\Gamma_0(4))$ is one dimensional for $k=0$ (spanned by the above theta function)  and is 0 for $k=1$ or $k<0$. 
	\end{theorem}
	Hence, we can see that in order to get non trivial spaces in these low half-integral weights, we have to allow our modular forms to have poles at cusps. Let's denote this new space of weakly holomorphic modular forms in the Kohnen plus-space by $M_{k+1/2}^{+,!}(\Gamma_0(4))$. An easy classification of the elements of $M_{k+1/2}^{+,!}$ is given by the following 
	\begin{proposition}
		$$f(\tau)\in M_{k+1/2}^{+,!}(\Gamma_0(4))\iff f(\tau)\Delta(4\tau)^n\in M_{k+12n+1/2}(\Gamma_0(4))$$ for some $n$.
	\end{proposition}
	Now we turn ourselves to the construction of an explicit basis in weight $1/2$ and $3/2$. Later we will look at groups of the form $\Gamma_0(4p)$ for $p|\#\MM$ where a more explicit method will be developed. 
	\begin{proposition}
		For every $d\equiv 0,3[4]$, there is a unique $f_d\in M^{!,+}_{1/2}(\Gamma_0(4))$ such that $$f_d(\tau)=q^{-d}+\sum_{D>0,D\equiv 0,1[4]}A(D,d)q^D$$ and the functions $f_0,f_3,f_4,f_7,...$ form a  basis of $M^{!,+}_{1/2}(\Gamma_0(4))$.
	\end{proposition}
	\begin{proof}
		For the uniqueness, suppose we have two such forms $f_d$ and $h_d$. Their difference is then $$f_d-h_d=\sum_{D>0,D\equiv 0,1 [4]}(A(D,d)_f-A(D,d)_h)q^D\in M_{1/2}^+(\Gamma_0(4))$$ Using Kohnen's isomorphism, we can see that $S_0=\{0\}$ and so $f_d-h_d=0\iff f_d=h_d$.\\
		For the existence, it suffices to construct them via the following algorithm. The idea is that we construct $f_0$ and $f_3$ by hand and the others by induction. $f_0$ is simply the $\theta$-function and $f_3$ can be constructed by use of the Rankin-Cohen bracket $$f_3=[\theta(\tau),E_{10}(4\tau)]/\Delta(4\tau)=q^{-3}-248q+26752q^4-...$$ To construct $f_d$ for $d\geq 4$, we take $f_{d-4}$ and multiply it by $j(4\tau)$ to get a form in the plus-space with leading coefficient $q^{-d}$. To kill the negative powers of $q$, we substract multiples of $f_{d'}$ for $0\leq d'<d$.
	\end{proof}
	The first basis elements are
	\begin{itemize}
		\item[1.]$f_0=1+2q+2q^4+2q^9+O(q^{16})$
		\item[2.]$f_3=q^{-3}-248q+26752q^4-85995q^5+O(q^8)$
		\item[3.]$f_4=q^{-4}+492q+143376q^4+565760q^5+O(q^8)$
		\item[4.]$f_7=q^{-7}-4119q+8288256q^4-52756480q^5+O(q^8)$
	\end{itemize}
	In a similar fashion we compute a basis for $M_{3/2}^{+,!}(\Gamma_0(4))$. 
	\begin{proposition}
		For every positive integer $D\equiv 0,1 [4]$, we define $g_D$ as the unique form in $M_{3/2}^{+,!}(\Gamma_0(4))$ with Fourier expansion of the form $$g_D(\tau)=q^{-D}+\sum_{d\geq 0, d\equiv 0,3[4]}B(D,d)q^d$$
	\end{proposition}
	\begin{proof}
		The proof for the uniqueness is exactly the same as above. The difference of two such functions lives in $M_{3/2}^+$ and since $M_2(\Gamma)$ is of dimension 0, we obtain uniqueness. For the existence, we define $g_1\triangleq g$ to be Zagier's weight $3/2$ modular form previously seen. For $g_4$ we again use the Rankin-Cohen bracket and get $$g_4=[g_1(\tau),E_{10}(\tau)]/\Delta(4\tau)$$ For higher forms $g_D$, $D>4$, we multiply $g_{D-4}$ by $j(4\tau)$ and substract multiples of $g_{D'}$ for $1\leq D'<D$. 
	\end{proof}
	\begin{example}
		The first basis elements are
		\begin{itemize}
			\item[1.] $g_1=q^{-1}-2+248q^3-492q^4+4119q^7+O(q^8)$
			\item[2.] $g_4=q^{-4}-2-26752q^3-143376q^4-8288256q^7+O(q^8)$
			\item[3.] $g_5=q^{-5}+0+85995q^3-565760q^4+52756480q^7+O(q^8)$
			\item[4.] $g_8=q^{-8}+0-1707264q^3-18473000q^4-5734772736q^7+O(q^8)$
		\end{itemize}
	\end{example}
	An interesting observation is that our bases $\{f_d\}_d$ and $\{g_D\}_D$ satisfy the relation $$A(1,d)=-B(1,d)$$ and the more general relation $$A(D,d)=-B(D,d)$$ for all $d,D$. This last equality is what we call the Zagier duality. We now give a proof of this statement that relies on ideas of Masanobu Kaneko
	\begin{proof}
		Let $f\in M_{1/2}^!(\Gamma_0(4))$, $g\in M_{3/2}^!(\Gamma_0(4))$. Their product $fg\in M_2^!(\Gamma_0(4))$ and by applying the operator $U_4$, we get $fg|_{U_4}\in M_2^!(\Gamma)$ (for more details, see theorem \ref{duality}). Any such function is a polynomial $P(j'(\tau))$ in the derivative of $j$. This implies that its constant term is 0. Hence by computing the constant term of $f_dg_D$ we get $$f_dg_D=q^{-d-D}+q^{-d}\sum_{d'\geq 0}B(D,d')q^{d'}+q^{-D}\sum_{D>0}A(D',d)q^{D'}+\sum_{d'\geq 0}B(D,d')q^{d'}\sum_{D'>0}A(D',d)q^{D'}$$ and thus  $$f_dg_D=B(D,d)+A(D,d)+h(q)$$ Now applying $U_4$ to $f_dg_D$ we obtain $$A(D,d)+B(D,d)=0$$
	\end{proof}
	Lets take a look now at the modified Hilbert polynomial $$\mathcal{H}_d(j(\tau))=\prod_{Q\in \Gamma\backslash Q_d}(j(\tau)-j(\alpha_Q))^{1/w_Q}$$ Finding an explicit $q$-product expansion for $\mathcal{H}_d(j(\tau))$ accounts to proving Borcherds theorem. We define a sequence of Faber polynomials in $j(\tau)$ as follow with the help of the following proposition 
	\begin{proposition}
		$\forall m\geq 0,\exists ! J_m(j(\tau))\in \ZZ[j], J_m\in M_0^!(\Gamma)$ such that $$J_m(j)=q^{-m}+O(q)$$
	\end{proposition}
	\begin{proof}
		Existence is easily justified by the fact that $\{j^n\}_n$ forms a basis of $M_0^!$. Thus, for $j(\tau)^m=q^{-m}+...+O(q)$, we can get rid of the negative powers of $q$ by substracting multiples of $j^{m'}$ for $m'<m$. This whole linear combination is $J_m$.\\
		For the uniqueness, note that the difference of two such forms would be a holomorphic cusp form of weight $0$ and thus the $0$ function.
	\end{proof}
	\begin{definition}
		The trace of $J_m(\tau)$ is $$Tr_m(d)=\sum_{Q\in \Gamma\backslash Q_d}\frac{1}{w_Q}J_m(\alpha_Q)$$
	\end{definition}
	These Faber polynomials will help us to rewrite $\mathcal{H}_d(j(\tau))$ in terms of these "higher traces"
	\begin{proposition}
		$\mathcal{H}_d(j(\tau))=q^{-H(d)}exp(-\sum_{m=1}^\infty Tr_m(d)q^m/m),\forall d$
	\end{proposition}
	\begin{proof}
		It suffices to prove the equality $$j(\tau)-j(z)=q^{-1}\exp{\Big(-\sum_{m=1}^\infty J_m(z)q^m/m\Big)}$$ with $\Im{(\tau)}$ having sufficiently large imaginary part. To do that, we note that the logarithmic derivative with respect to $\tau$ of either side is a $\Gamma$-invariant meromorphic function (for $\Im{(\tau)}$ large enough) of $z$ whose only pole in the fundamental domain is a simple pole of residue $1$ at $z=\tau$ which vanishes at $\infty$.
	\end{proof}
	To explicitly compute these higher traces, we will generalize the functions $f_d$ and $g_D$ constructed before
	\begin{definition}{\cite{Jen}\\}
		Let $f=\sum a(n)q^n\in M_{k+1/2}^{!,+}$ and $p$ an odd prime. Then the half-integral weight Hecke operator $T'(p^2)$ maps $f$ to a modular form of same weight given by $$f(\tau)|_{k+1/2}T'(p^2)=\sum\Big(a(p^2n)+\Big(\frac{(-1)^k}{n}\Big)\Big(\frac{n}{p}\Big)p^{k-1}a(n)+p^{2k-1}a(\frac{n}{p^2})\Big)$$This formula also holds for $p=2$ if we take $n/2=0$ if $n$ is even and $(-1)^{(n^2-1)/8}$ if $n$ is odd. In the case of $k\leq0$, this formula introduces nontrivial denominators so we normalize by multiplying by $p^{1-2k}$ so that our forms will still have integer coefficients, giving $$T(p^2)=\left\{
		\begin{array}{ll}
		p^{1-2k}T'(p^2), & k\leq 0 \\
		T'(p^2), & \text{ otherwise }
		\end{array}
		\right.$$
	\end{definition}
	For any $m\geq 0$, we apply $|_{1/2}T(m^2)$ to $f_d$ and $|_{3/2}T(m^2)$ to $g_D$ and define $A_m(D,d)$, $B_m(D,d)$ to be the coefficients of $f_d|_{1/2}T(m^2)$ and $g_D|_{3/2}T(m^2)$ respectively. These are integers and using the definition of the Hecke operators, we can compute $A_p(D,d)=pA(p^2D,d)+\frac{D}{p}A(D,d)+A(p^{-2}D,d)$ and $B_p(D,d)=B(D,p^2d)+\frac{-d}{p}B(D,d)+pB(D,p^{-2}d)$ (with the convention that $A(p^{-2}D,d)=0$ unless $p^{-2}D$ is an integer congruent to $0,1[4]$ by the plus-space condition and the same applies for $B(D,p^{-2}d)$). Letting $m$ arbitrary and $D=1$, we get $$A_m(1,d)=\sum_{n|m}nA(n^2,d)=-B_m(1,d),\forall d$$ and again we may ask if this relation is true for all $D$. In fact it has a generalization for fundamental discriminants $D$ given by $$A_m(D,d)=\sum_{n|m}n\Big(\frac{D}{m/n}\Big)A_1(n^2D,d)=-B_m(D,d)$$ We summarize the above facts in the following theorem
	\begin{lemma}{\cite{Za}}
		\begin{itemize}
			\item[1.] $\mathcal{H}_d(j(\tau))=q^{-H(d)}\exp{\Big(-\sum_{m=1}^\infty Tr_m(d)q^m/m\Big)},\forall d$.
			\item[2.] $Tr_m(d)=-B_m(1,d),\forall m,d$.
			\item[3.] $A_m(D,d)=-B_m(D,d),\forall m,D,d$.
		\end{itemize}
	\end{lemma}
	Now we are ready to construct our $q$-product expansion for the modified Hilbert class polyomial
	\begin{proof}{(Borcherds)}
		$$\mathcal{H}_d(j(\tau))=q^{-H(d)}\exp(-\sum_{m=1}^\infty Tr_m(d)q^m/m)$$ $$=q^{-H(d)}\exp(\sum_{m=1}^\infty B_m(1,d)q^m/m)$$ $$=q^{-H(d)}\exp(\sum_{m=1}^\infty -A_m(1,d)q^m/m)$$ $$=q^{-H(d)}\exp(\sum_{m=1}^\infty \sum_{n|m}-nA(n^2,d)q^m/m)$$ and by exchanging order of summation $$=q^{-H(d)}\exp(\sum_{n\geq 1,m\geq 1}-nA(n^2,d)q^{mn}/mn)$$ $$=q^{-H(d)}\exp(\sum_{n\geq 1,m\geq 1}-A(n^2,d)(q^n)^m/m)$$ $$=q^{-H(d)}\exp(\sum_{n\geq 1}A(n^2,d)\cdot -\sum_{m\geq 1}(q^n)^m/m)$$ $$=q^{-H(d)}\exp(\sum_{n\geq 1}A(n^2,d)\log(1-q^n))$$ $$=q^{-H(d)}\prod_{n=1}^\infty \exp(A(n^2,d)\log(1-q^n))$$ $$=q^{-H(d)}\prod_{n=1}^\infty\exp(\log((1-q^n)^{A(n^2,d)}))$$ $$=q^{-H(d)}\prod_{n=1}^\infty(1-q^n)^{A(n^2,d)}$$
		where the exponents, as described by Borcherds theorem, are those of weight $1/2$ modular forms.
	\end{proof}
	\pagebreak
	
	\section{Traces of Singular Moduli for the Fricke Group}
	There are various directions in which mathematicians have extended Zagier's work on singular moduli. We will explore one of these paths by studying the Fricke group.\\
	Let $p$ be a prime in the set $\{2,3,5,7,11,13,17,19,23,29,31,41,47,59,71\}$, $W_p\triangleq \m{0&-1/\sqrt{p}\\\sqrt{p}&0}$ be the Fricke involution and $\Gamma_0(p)^*\triangleq \langle \Gamma_0(p),W_p \rangle<SL_2(\RR)$ be the Fricke group acting on $\HH$ in the usual way. The above list of primes is not random, it consists of these values of $p$ for which the modular curve $X_0(p)^*=\Gamma_0(p)^*\backslash \HH\cup \PP_1(\QQ)$ has genus $0$. More precisely, we have the famous theorem 
	\begin{theorem}
		$$p|\#\MM\iff X_0(p)^*\text{ has genus }0$$ where $\MM$ denotes the Monster group.
	\end{theorem}
	The idea is that we would like to extend Zagier's work in a fairly natural way by staying on a genus 0 modular curve. The following proposition gives us an explicit description of the elements of the Fricke group 
	\begin{proposition}
		Let $A\in \Gamma_0(p)^*$, then  $$A=\m{a\sqrt{e}&b/\sqrt{e}\\cp/\sqrt{e}&d\sqrt{e}}$$ with $a,b,c,d,e\in \ZZ,det(A)=ade-cpb/e=1,e|p,e>0$.
	\end{proposition}
	\begin{proof}
		The fact that $e|p,e>0\implies e=1,e=p$. The case $e=1$ is trivial so look at $e=p$.\\
		First we observe that $W_p^2=-I$ where $I$ is the identity matrix. It follows that $W_p^4=I$ with $W_p^{-1}=W_p^3$. Thus, it suffices to consider only the product of an element $\gamma=\m{A&B\\Cp&D}\in \Gamma_0(p)$ and $W_p$ since $\gamma W_p^2=-\gamma, \gamma W_p^3=-\gamma W_p,\gamma W_p^4=\gamma$ (product in the other direction follows exactly the same steps so we don't repeat it). Hence $$\gamma W_p=\m{B\sqrt{p}&-A/\sqrt{p}\\Dp/\sqrt{p}&-C\sqrt{p}}$$ which is of the desired form.  
	\end{proof}
	In order to study traces of singular moduli for the Fricke group $\Gamma_0(p)^*$, we need to construct hauptmoduln for each prime $p$. We will explore how they can be constructed in a fairly easy way from $\eta$-quotients for $\Gamma_0(p),\,\,p=2,3,5,7,13$, those prime values for which $gen(X_0(p))=0$.
	\begin{definition}
		Let $\tau\in \HH$ and $\zeta_{24}=e^{2\pi i/24}$. Then the Dedekind Eta function is defined as $$\eta(\tau)=q^{1/24}\prod_{n\geq 1}(1-q^n)$$ and has the following properties
		\begin{itemize}
			\item[1.]$\eta\neq 0$ in $\HH$ 
			\item[2.]$\eta$ has a zero of order $1/24$ at $\infty$
			\item[3.]$\eta(\gamma\tau)=\psi(\gamma)(c\tau+d)^{1/2}\eta(\tau),\,\forall\gamma\in SL_2(\ZZ)$ where $\psi$ is some multiplier system.
		\end{itemize}
	\end{definition}
	\begin{definition}
		Let $G<SL_2(\RR)$ be a subgroup of genus 0. A Hauptmodul $f$ of $G$ is
		\begin{itemize}
			\item[1.] a meromorphic function on $G\backslash \HH\cup \PP_1(\QQ)=X_G$
			\item[2.] a generator of the function field on $X_G$ over $\CC$ or equivalently, a uniformizer $$f:X_G\cong \PP_1(\CC),\tau\mapsto f(\tau)$$
		\end{itemize}
		If furthermore, $f$ has a q-expansion of the form $$f(\tau)=q^{-1}+\sum_{n\geq 1}a_nq^n$$ with $a_n\in \ZZ,\,\forall n\geq 1$, then $f$ is called a canonical hauptmodul of $G$.
	\end{definition}
	The method goes as follow. We first start with a form of weight $0$ on $\Gamma_0(p)$ and then look at its transformation equation under $W_p$. Then the trick is to complete this form with another form of weight $0$ on $\Gamma_0(p)$ so that we obtain the desired modular transformation. This might seem vague but I hope the following lines will clarify the situation. \\
	Consider the Dedekind $\eta$-function of weight $1/2$ on $SL_2(\ZZ)$. Then $\eta(p\tau)$ is a modular form of weight $1/2$ on $\Gamma_0(p)$ and the $\eta$-quotients $\eta(p\tau)/\eta(\tau)$, $\eta(\tau)/\eta(p\tau)$ are modular forms of weight $0$ on $\Gamma_0(p)$. By taking $W_p=\m{0&-1/\sqrt{p}\\\sqrt{p}&0}$ we can compute its action on the function $\eta(m\tau)$ $$\eta(mW_p\tau)=\eta(m\frac{-1}{p\tau})=\eta(\frac{-1}{(p/m)\tau})=\sqrt{-i\frac{p}{m}\tau}\eta(\frac{p}{m}\tau)$$ and so for the quotient $u(\tau)=\Big(\frac{\eta(p\tau)}{\eta(\tau)}\Big)^a$ for some $a\in \ZZ$, we get $$u(W_p\tau)=\Big(\frac{\eta(pW_p\tau)}{\eta(W_p(\tau))}\Big)^a=\Big(\frac{\sqrt{-i\tau}\eta(\tau)}{\sqrt{-ip\tau}\eta(p\tau)}\Big)^a=p^{-a/2}\Big(\frac{\eta(\tau)}{\eta(p\tau)}\Big)^a$$ Similarly, if we look at the quotient $v(\tau)=\Big(\frac{\eta(\tau)}{\eta(p\tau)}\Big)^a$, we get 
	$$v(W_p\tau)=p^{a/2}\Big(\frac{\eta(p\tau)}{\eta(\tau)}\Big)^a$$ Both functions are inverses under the action of $W_p$ and in fact $$u(W_p\tau)=p^{-a/2}v(\tau),v(W_p\tau)=p^{a/2}u(\tau)$$ Putting $u$ and $v$ together, we can see that $$f_a(\tau)\triangleq p^{a/2}u(\tau)+v(\tau)=p^{a/2}\Big(\frac{\eta(p\tau)}{\eta(\tau)}\Big)^a+\Big(\frac{\eta(\tau)}{\eta(p\tau)}\Big)^a$$ transforms under $W_p$ like a modular function $$f_a(W_p\tau)=p^{a/2}u(W_p(\tau))+v(W_p(\tau))=p^{a/2}p^{-a/2}v(\tau)+p^{a/2}u(\tau)=f_a(\tau)$$ It follows that $f$ is a modular form of weight $0$ under $\Gamma_0(p)^*=\langle\Gamma_0(p),W_p\rangle$. Now we would like to normalize $f_a$ so that it has a pole of order $1$ at infinity. To do so, we compute the value of $a$ for a given $p$. It is easy to understand the behaviour of $f_a$ at $\infty$ since the only term in $f_a$ that will affect the behaviour as $q\to 0$ is $\Big(\frac{\eta(\tau)}{\eta(p\tau)}\Big)$ which behaves like $q^{a(1-p)/24}$ and because we want a pole of order $1$, we require that $\frac{a(1-p)}{24}=-1$ and this gives $a=\frac{-24}{1-p}$.\\
	Now we are left with removing the constant term of $f_a$ to get a canonical hauptmodul. This can be done by numerical computations using MAGMA for example. \\
	Here are the hauptmoduln for $p=2,3,5,7,13$
	\begin{itemize}
		\item[2.] $j_2(\tau)=2^{12}\Big(\frac{\eta(2\tau)}{\eta(\tau)}\Big)^{24}+\Big(\frac{\eta(\tau)}{\eta(2\tau)}\Big)^{24}+24$
		\item[3.]$j_3(\tau)=3^6\Big(\frac{\eta(3\tau)}{\eta(\tau)}\Big)^{12}+\Big(\frac{\eta(\tau)}{\eta(3\tau)}\Big)^{12}+12$
		\item[5.]$j_5(\tau)=5^3\Big(\frac{\eta(5\tau)}{\eta(\tau)}\Big)^{6}+\Big(\frac{\eta(\tau)}{\eta(5\tau)}\Big)^{6}+6$
		\item[7.]$j_7(\tau)=7^2\Big(\frac{\eta(7\tau)}{\eta(\tau)}\Big)^{4}+\Big(\frac{\eta(\tau)}{\eta(7\tau)}\Big)^{4}+4$
		\item[13.]$j_{13}(\tau)=13\Big(\frac{\eta(13\tau)}{\eta(\tau)}\Big)^{2}+\Big(\frac{\eta(\tau)}{\eta(13\tau)}\Big)^{2}+2$
	\end{itemize}
	The method we have described so far worked only when the original group was of genus 0. For other $p's$, the functions $f_a$ will no longer define isomorphisms but multiple covers. To obtain a general method for computing hauptmoduln, we turn ourselves towards the theory of Poincaré series and Rademacher sums. We will follow in details the article of Miranda Chang and John Duncan \cite{MCJD}. \\
	What are the natural ways to obtain modular forms ? Or more generally, what are the natural ways to construct functions satisfying certain symmetries given by a set $\Gamma$ ? One could for example pick a function $g$ and define the symmetric function $f$ as $$f(\tau)=\sum_{\gamma\in \Gamma}g(\gamma\tau)$$
	by summing over all the images of $g$. If $\Gamma$ is a finite set then $f$ is invariant under the group of symmetry. However, if $\Gamma$ is infinite, we might have convergence problems. One way to go around this problem was pioneered by Poincaré and consists to start with a function $g$ already invariant under a large portion of the group and then sum over cosets representatives $$f(\tau)=\sum_{\gamma\in \Gamma_g\backslash \Gamma}g(\gamma\tau)$$ where $\Gamma_g$ is the set fixing $g$. Again one would have to be careful about the convergence of the series. As an example, suppose we want to construct a modular form of weight $w=2k$ over $\Gamma=SL_2(\ZZ)$. Using Poicaré's idea, we consider the function $g(\tau)=e^{2\pi im\tau}=e(m\tau)$. Then the subgroup of $\Gamma$ fixing $g$ is the stabilizer of the cusp at $\infty$ $$\Gamma_\infty=\Big\{\m{1&b\\0&1}:b\in \ZZ\Big\}$$
	Recall that a holomorphic function $f$ on $\HH$ is a modular form of weight $w$ if it satisfies the equation $f|_w(\Gamma)=f\,\,\forall\gamma\Gamma$. Summing over all the images of $g$ under the slash-$w$ operator leads us to consider $$f(\tau)=\sum_{\gamma\in \Gamma_\infty\backslash\Gamma}g|_w(\gamma\tau)=\sum_{\gamma\in\Gamma_\infty\backslash \Gamma} e(m\frac{a\tau+b}{c\tau+d})(c\tau+d)^{-w}$$ where $j(\gamma,\tau)=c\tau+d$ is the automorphy factor satisfying $j(\gamma,\gamma'\tau)=j(\gamma\gamma',\tau)/j(\gamma',\tau),\gamma\in SL_2(\ZZ)$. When $k>1$, the above is holomorphic in $\HH$ and for $k\leq 1$, the sum is not absolutely convergent. In fact, it is not even conditionnaly convergent for $k<1/2$. We can now ask if there are ways to regularize the above sum to overcome the convergence problems. These regularizations of Poicaré series are refered to as Rademacher sums. We now investigate their properties in view of constructing hauptmoduln for $\Gamma_0(p)^*$.\\
	Recall that $SL_2(\RR)=Aut(\HH)$ is the automorphism group of the upper half-plane $\HH$ where the action is via fractional linear transformations $\gamma\cdot\tau=\frac{a\tau+b}{c\tau+d}$. Let $j(\gamma,\tau)=(c\tau+d)^{-2}$ be the derivative of this action so that $(c,d)$ is the lower row of $\gamma$. Now let $\Gamma<SL_2(\RR)$ denote a subgroup of $SL_2(\RR)$ that contains $\pm I$ and that is commensurable with $SL_2(\ZZ)$. Let $\Gamma_\infty<\Gamma$ be the stabilizer subgroup of the cusp at infinity. Define the translation matrices $T=\m{1&1\\0&1}$ and $T^h=\m{1&h\\0&1}$. We can find a unique $h>0$ such that $\Gamma_\infty=\langle -I,T^h\rangle$. This $h$ is called the width of $\Gamma$ at $\infty$. In the case of $\Gamma_0(N)$ for example, this width is $h=1$. \\
	For $w\in \RR$, we say that the function $\psi:\Gamma\to\CC$ is a multiplier system of weight $w$ if $$\psi(\gamma_1)\psi(\gamma_2)j(\gamma_1,\gamma_2\tau)^{w/2}j(\gamma_2,\tau)^{w/2}=\psi(\gamma_1\gamma_2)j(\gamma_1\gamma_2,\tau)^{w/2}$$ for all $\gamma_1,\gamma_2\in \Gamma$ where we fix the principal branch of the logarithm in order to define the exponential $x\mapsto x^s$ for $s$ not an integer. With this multiplier system, we will say that a holomorphic function $f$ in $\HH$ is a modular form of weight $w$ and multiplier system $\psi$ for $\Gamma$ if $f|_{\psi,w}(\gamma)=f\,\,\forall\gamma\in\Gamma$. Remark that since $-\gamma\tau=\gamma\tau$, we have that $f|_{\psi,w}(-I\tau)=\psi(-I)(-1)^{-w}f(\tau)=f(\tau)$ and thus $\psi(-I)(-1)^{-w}=1$. It follows that $\psi(-I)=e(w/2)$. This condition is the consistency condition our multiplier system must satisfy in order that the corresponding space of weakly holomorphic modular forms is non vanishing. \\
	$\Gamma$ is assumed to be commensurable with $SL_2(\ZZ)$ so its natural action on the boundary of $\HH$ restricts to $\PP_1(\QQ)$. As always, the orbits of $\PP_1(\QQ)$ are called cusps of $\Gamma$ and if the quotient space $X_\Gamma=\Gamma\backslash(\HH\cup \PP_1(\QQ))$ is a genus zero surface, we say that $\Gamma$ is a genus 0 group. This is the case for $\Gamma_0(p)^*$.\\
	If $\Gamma$ has width $h$ at inifnity, then any multiplier system $\psi$ for $\Gamma$ restricts to a character on $\langle T^h \rangle<\Gamma_\infty$ and so we have $\psi(T^h)=e(\alpha)$ for a unique $\alpha\in \RR$, $0\leq \alpha<1$. It follows that the natural $\Gamma_\infty$ invariant function to consider for the slash-$(\psi,w)$ action is $q^\mu=e(\mu\tau)$ for $\mu$ such that $h\mu+\alpha\in \ZZ$. Summing the images of $g$ under the slash operator for cosets representatives gives $$P^{[\mu]}_{\Gamma,\psi,w}(\tau)=\sum_{\gamma\in\Gamma_\infty\backslash\Gamma}e(\mu\gamma\tau)\psi(\gamma)j(\gamma,\tau)^{w/2}$$ This series is called a \textit{Poicaré Series of weight $w$ and index $\mu$ for $\Gamma$ with multiplier $\psi$}.\\
	As said before, the case $w\leq 2$ gives in general divergent series and so we may ask if there is a natural way to regularize such sums. Since we are interested in hauptmoduln, we will focus on the case $w=0$.\\
	In order to regularize the above Poincaré series, we need to modify the terms in the series and the order of summation. We write $\gamma\infty$ for the limit of $\gamma\tau$ as $\tau\to i\infty$ so $\gamma\infty=a/c$ for $c\neq 0$ and is undefined otherwise. Now, for $w<1$, define the Rademacher sums $$R^{[\mu]}_{\Gamma,\psi,w}(\tau)=\delta_{a,0}\frac{1}{2}c_{\Gamma,\psi,w}(\mu,0
	)+\lim_{K\to\infty}\sum_{\gamma\in \Gamma_\infty\backslash \Gamma_{K,K^2}}e(\mu\gamma\tau)r_w^{[\mu]}(\gamma,\tau)\psi(\gamma)j(\gamma,\tau)^{\frac{w}{2}}$$ with $$r_w^{[\mu]}(\gamma,\tau)=\frac{1}{\Gamma(1-w)}\gamma(1-w,2\pi i\mu(\gamma\tau-\gamma\infty))$$ $$\Gamma(s)=\int_{0}^\infty t^{s-1}e^{-t}dt,\,\,s\in \RR_{>0}$$ $$\gamma(s,x)=\frac{\Gamma(s)}{e^x}\sum_{n\geq 0} \frac{x^{n+s}}{\Gamma(n+s+1)},\,\,s\in\RR_{>0},x\in \CC$$ $$c_{\Gamma,\psi,w}(\mu,0)=\frac{1}{h}e(-w/4)\frac{(2\pi)^{2-w}(-\mu)^{1-w}}{\Gamma(2-w)}\lim_{K\to \infty}\sum_{\gamma\in \Gamma_\infty\backslash \Gamma_K^\times/\Gamma_\infty}\frac{e(\mu \gamma\infty)}{c(\gamma)^{2-w}}\psi(\gamma)$$
	$$\Gamma_K^\times =\Big\{\m{a&b\\c&d}:0<|c|<K\Big\}$$ $$\Gamma_{K,K^2}=\Big\{\m{a&b\\c&d}:|c|<K,|d|<K^2\Big\}$$
	and $c(\gamma)>0$ the $c$ entry of $\gamma$. Furthermore, the above conditions leads to consider only $\mu\leq 0$
	since otherwise, $\tau\to 2\pi i\mu(\gamma\tau-\gamma\infty)$ covers the left-half plane and $r_w^{\mu}(\gamma,\tau)$ can fail to be continuous with respect to $\tau$. Here are results about convergence in the case of weight 0 
	\begin{theorem}
		Let $\Gamma<SL_2(\RR)$ commensurable with the modular group $SL_2(\ZZ)$ and containing $-I$. Then the Rademacher sum $R^{[\mu]}_{\Gamma,1,0}$ converges locally uniformly for $\tau\in \HH$ and for any negative integer $\mu$. Furthermore, the Rademacher sum is a weak mock modular form with shadow $$S_{\Gamma,1,0}^{[\mu]}(\tau)=\frac{(-\mu)}{\Gamma(1)}R_{\Gamma,1,2}^{[-\mu]}$$
	\end{theorem}
	Recall that the cusp forms of weight 2 with trivial multiplier for $\Gamma$ are in correspondence with holomorphic 1-forms on the Riemann surface $X_\Gamma$ and the dimension of the space of holomorphic 1-forms on a Riemann surface is equal to its genus. So if $\Gamma$ has genus zero, then $X_\Gamma$ has no non-zero 1-forms and we must have $S_{\Gamma,1,0}^{[-1]}=0$. It follows that a Rademacher Sum on a genus zero group is a weak mock modular form with trivial shadow and thus a weakly holomorphic modular form.\\
	In order to describe the Fourier expansion of Rademacher Sums, we introduce Rademacher series. Let $\mu,\nu\in \frac{1}{h}(\ZZ-\alpha)$, then $$c_{\Gamma,\psi,w}(\mu,\nu)=\frac{1}{h}\lim_{K\to\infty}\sum_{\gamma\in \Gamma_\infty\backslash \Gamma_K^\times/\Gamma_\infty} K_{\gamma,\psi}(\mu,\nu)B_{\gamma,w}(\mu,\nu)$$ where $\Gamma_K^\times$ is as before and $K_{\gamma,\psi}(\mu,\nu)=e(\mu\frac{a}{c})e(\nu\frac{d}{c})\psi(\gamma)$ with $B_{\gamma,w}(\mu,\nu)=e(\frac{-w}{4})\sum_{k\geq 0}(\frac{2\pi}{c})^{2k+2-w}\frac{(-\mu)^{k+1-w}}{\Gamma(k+2-w)}\frac{\nu^k}{k!}$ for $w\leq 1$, $\gamma=\m{a&b\\c&d},\,c>0$.
	\begin{theorem}
		Let $\Gamma<SL_2(\RR)$ a subgroup commensurable with $SL_2(\ZZ)$ and containing $-I$. Then the Rademacher series $c_{\Gamma,1,0}(\mu,\nu)$ and $c_{\Gamma,1,2}(\mu,\nu)$ converge for all $\mu,\nu\in \ZZ$.
	\end{theorem}
	\begin{theorem}{(Structure Result)\\}
		Let $\Gamma<SL_2(\RR)$ a subgroup commensurable with $SL_2(\ZZ)$ and containing $-I$. Then the Fourier expansion of the Rademacher sum of weight $0$, trivial multiplier on $\Gamma$ and index $\mu$ is  $$R_{\Gamma,1,0}^{[\mu]}(\tau)=q^\mu+\sum_{h\nu+\alpha\in\ZZ,\,\nu\geq 0}c_{\Gamma,1,0}(\mu,\nu)q^\nu$$ with $c_{\Gamma,1,0}(-1,0)$ as before.
	\end{theorem}
	There is an interesting relation with Monstrous Moonshine that gives us the desired characterization of our hauptmoduln in terms of McKay-Thompson Moonshine series type
	\begin{theorem}
		If $\Gamma$ is of Moonshine type, then $$R^{[-1]}_{\Gamma,1,0}:X_\Gamma\cong \PP_1(\CC)$$ and the canonical hauptmodul of $X_\Gamma$ is given by $T(\tau)=R^{[-1]}_{\Gamma,1,0}(\tau)-c_{\Gamma,1,0}(-1,0)$ 
	\end{theorem}
	Let $p|\#\MM$ and consider $\Gamma_0(p)^*$. The associated hauptmodul is given by $$j_p(\tau)=R_{\Gamma_0(p)^*,1,0}^{[-1]}(\tau)-c_{\Gamma_0(p)^*,1,0}(-1,0)$$ To explicitly compute these functions, we need to find representatives for the double cosets involved in the definition. To do so, we first remark that $\Gamma_0(p)^*=\Gamma_0(p)\sqcup\Gamma_0(p)W_p$ where this union is a disjoint one. Also, it is easily seen that $\Gamma_0(p)^*_\infty=\Gamma_0(p)_\infty$. Hence, to describe the first coset, we let $\gamma\in \Gamma_0(p)$ with $c\neq 0$ and $\gamma_m=\m{1&m\\0&1}$. Then $$\pm \gamma_m\gamma\gamma_n=\pm\m{a+mcp&n(a+mcp)+b+md\\cp&ncp+d}$$ We can see that for $c\geq 1$ and $d [cp]$, we have $a$ uniquely determined by $ad\equiv 1[cp]$ and $b$ uniquely determined by $b=\frac{ad-1}{cp}$. It follows that the first double coset is given by $$\bigcup_{c=1}^\infty\bigcup^{cp}_{d=1,\,(d,cp)=1,\,ad\equiv 1[cp]}\Gamma_0(p)_\infty \m{a&*\\cp&d}\Gamma_0(p)_\infty$$ 
	For the second coset, let $\gamma=\m{a&-b\\-cp&d}$. Then $$\pm\gamma_m\gamma W_p\gamma_n=\m{(md-b)\sqrt{p}&n(md-b)\sqrt{p}-\frac{a}{\sqrt{p}}+mc\sqrt{p}\\d\sqrt{p}&(nd+c)\sqrt{p}}$$
	Letting $d\geq 1$ with $(d,p)=1$ and $b[d]$, we can see that $c$ is determined by $-bcp\equiv 1[d]$ and $a$ is determined by $a=\frac{bcp+1}{d}$. This gives us for the coset representatives $$\bigcup_{d=1,(d,p)=1}^\infty\bigcup_{c=1,(c,d)=1,-bcp\equiv 1[d]}^d\Gamma_0(p)_\infty \m{-b\sqrt{p}&*\\d\sqrt{p}&c\sqrt{p}}\Gamma_0(p)_\infty$$
	Note that in order to use the definition of the Bessel function given above, we had to pick our representatives to have lower left entry positive which explained our choice up here. 
	\begin{theorem}
		The hauptmodul for $\Gamma_0(p)^*$ is $$j_p(\tau)=R^{[-1]}_{\Gamma,1,0}-c_{\Gamma,1,0}(-1,0)=q^{-1}+\sum_{\nu\geq 1}c_{\Gamma,1,0}(-1,\nu)q^\nu$$ with $$c_{\Gamma,1,0}(-1,\nu)=\sum_{c=1}^\infty\sum_{d=1,(d,cp)=1,ad=1[cp]}^{cp}e(-\frac{a}{cp})e(\nu\frac{d}{cp})\sum_{k\geq 0}(\frac{2\pi}{cp})^{2k+2}\frac{\nu^k}{\Gamma(k+2)k!}+$$ $$\sum_{d=1,(d,p)=1}^\infty \sum_{c=1,(c,d)=1,-bcp\equiv 1[d]}^d e(\frac{b}{d})e(\nu\frac{c}{d})\sum_{k\geq 0}(\frac{2\pi}{d\sqrt{p}})^{2k+2}\frac{\nu^k}{\Gamma(k+2)k!}$$ $$=\sum_{c=1}^\infty K(\nu,-1,cp)\Big (\frac{2\pi I_1(\frac{4\pi\sqrt{\nu}}{cp})}{c\sqrt{\nu}p}\Big)+\sum_{d=1,(d,p)=1}^\infty K(\nu,-1,d)\Big (\frac{2\pi I_1(\frac{4\pi\sqrt{\nu}}{d\sqrt{p}})}{d\sqrt{\nu p}}\Big)$$ where $K$ is the Kloosterman sum and $I_1$ the modified Bessel function of the first kind.
	\end{theorem}
	We now present a generalization of Zagier's work on singular moduli for $\Gamma_0(p)^*$. We would like to study the quotient of some set of quadratic forms by $\Gamma_0(p)^*$ in order to obtain an object in bijection with $\Gamma\backslash Q_d$, $\Gamma=SL_2(\ZZ)$, so that we can define in an analogous way $H(d)$ and $\bold{t}(d)$. We let $d\geq 0,-d\equiv \square[4p]$ and we consider roots of quadratic forms $Q(X,Y)=[a,b,c]$ with $a\equiv 0[p]$ and $(a,b,c,p)=1$. Replacing $a$ by $ap$ in the above leads to $$Q(X,Y)=apX^2+bXY+cY^2$$ with discriminant $-d=b^2-4apc$ which gives $-d\equiv \square [4p]$. The equivalence class of a root $\alpha_Q$ for $\Gamma_0(p)$ has another invariant apart from its equivalence class in $\Gamma\backslash Q_d$, it is the value of $b[2p]$. If we fix this class, i.e. a $\beta[2p]$ with $-d\equiv \beta^2[4p]$ and consider $Q=[ap,b,c]$ with $b\equiv \beta[2p]$, then the number of $\Gamma_0(p)$-equivalence class of such a $Q$ counted with multiplicity $1/w_Q$ is $H(d)$. Moreover, $$\Gamma_0(p)\backslash Q_{d,p,\beta}\cong \Gamma\backslash Q_d$$
	However, since $j_p$ is the hauptmodul of $\Gamma_0(p)^*$, we can get rid of the $\beta$-dependence by restricting to the quotient of $\HH$ by $\Gamma_0(p)^*$ since the trace is independent of $\beta$ there. Hence, by restricting to this case, we get $$\Gamma_0(p)^*\backslash Q_{d,p}\cong \Gamma\backslash Q_d$$
	It follows that we can write in our setting $$\bold{t}(d)=\sum_{Q\in \Gamma_0(p)^*\backslash Q_{d,p}}\frac{j_p(\alpha_Q)}{w_Q}$$ with $$Q_{d,p}=\{[ap,b,c] : -d=b^2-4apc,-d\equiv \square [4p]\}$$
	We have that $\Gamma_0(p)^*\acts Q_{d,p}$ in the usual way with $W_p\cdot Q(X,Y)=[cp,-b,a]$. Now, we are interested in studying the generating function of the trace and relating it to a modular form of weight $3/2$ on $M_{3/2}^{!,+}(\Gamma_0(4p))$. A first observation is that $$\sum_{d\geq 1,-d\equiv \square[4p]}\bold{t}(d)q^d$$ can be rewritten as $$\sum_{4np-r^2=d\geq 1}\bold{t}(4np-r^2)q^d$$ and remembering the beautiful isomorphism \cite{ZaEi}
	$$J_{2,p}(\Gamma^J)\cong M^{+}_{3/2}(\Gamma_0(4p))$$
	we can see that the natural language to study the above generating function would be the language of Jacobi Forms since Jacobi forms of even weight and prime index have Fourier coefficients $c(n,r)$ depending only on $4np-r^2$. 
	The idea here will be to construct an appropriate Jacobi form and to project it into the Kohnen plus-space via identification of Fourier coefficients, namely $\sum_{n,r}c(n,r)q^n\zeta^r\mapsto \sum_dc(d)q^d,\,d=4np-r^2$. First some preliminary results
	\begin{proposition}
		$\exists ! \phi_{D,p}\in J_{2,p}^!(\Gamma^J)$ with Fourier coefficients $B(D,4np-r^2)$ satisfying $B(D,-D)=1$, $B(D,d)=0$ for $d=4pn-r^2<0$ and $d\neq -D$. 
	\end{proposition}
	\begin{proof}
		Suppose there exists two such forms, then their difference would be a holomorphic Jacobi form of weight 2 and index $p$. Following the fact that $$dim(J_{2,p}(\Gamma^J))\cong dim(M_2(\Gamma_0(p)^*))$$ and according to the table page 132 \cite{ZaEi}, we get $dim(M_2(\Gamma_0(p)^*))=0$.
	\end{proof}
	Let $D=1$, then from the isomorphism given above we get that $\phi_{1,p}\mapsto q^{-1}-2+\sum_{d\geq 1,-d=\square[4p]}B(1,d)q^d\in M_{3/2}^{!,+}(\Gamma_0(4p))$ and by the above, such a form is unique. More generally, if we denote by $g_{D,p}$ the modular form correpsonding to $\phi_{D,p}$, we have that $$g_{D,p}(\tau)=q^{-D}+B(D,0)+\sum_{d\geq 1,-d=\square[4p]}B(D,d)q^d$$ is the unique such form with $B(D,0)=-2$ if $D$ is a square and $0$ otherwise.\\
	To express canonically all the $\phi_{D,p}\in J_{2,p}^!(\Gamma^J)$ and thus simultaneously showing their existence (and so the existence of the $g_{D,p}$), we recall that $J_{even,*}^!$ is the free polynomial algebra over $M^!_*(\Gamma)=\CC[E_4,E_6,\Delta^{-1}]$ on two generators $$a=\tilde{\phi_{-2,1}}=\phi_{10,1}/\Delta,\, b=\tilde{\phi_{0,1}}=\phi_{12,1}/\Delta$$ with $\phi_{k,1}\in J_{k,1}^{cusp}$.
	This implies that $\phi_{D,p}\in J^!_{2,p}(\Gamma)$ can be written as a linear combination of the form $a^ib^{p-i}$ over $M_*^!$ and more precisely as $$\phi_{D,p}(\tau,z)=\sum_{\nu=1}^p f_{p,\nu}(\tau)a^\nu b^{p-\nu}$$
	with $f_{p,\nu}\in M^!_{2\nu +2}(SL_2(\ZZ))$ having constant term $\frac{(-1)^{\nu-1}}{12^{p-1}}{p-1 \choose \nu-1}$.
	\begin{theorem}{\cite{Za}}
		For $p=2,3,5$, we have $$\bold{t}(j_p,d)=-B^{(p)}(1,d),\, \forall d$$ with $B^{(p)}(1,d)$ the coefficients $g_{1,p}\in M_{3/2}^!(\Gamma_0(4N))$.
	\end{theorem}
	We are then left with showing that for the other values of $p$ dividing the order of the monster group, we still have that the generating series of traces of singular moduli is a modular form of weight $3/2$. We will prove this by applying Borcherds'theorem about automorphic forms with singularities on Grassmannian to $\Gamma_0(p)^*$ but before that we investigate if the Zagier duality between weight $3/2$ and $1/2$ still exists in level $4p$. Let $f_{d,p}\in M_{1/2}^{!,+}(\Gamma_0(4p))$ be such that $$f_{d,p}(\tau)=q^{-d}+\sum_{D>0,D=\square [4p]}A(D,d)q^D$$ We claim that such a form is unique and to prove that, we first prove the duality and then from the uniqueness of the $g_{D,p}$ and their existence given by the structure theorem of Jacobi forms, we will have that such $f_{d,p}$'s exist and are unique.
	\begin{theorem}
		\label{duality}
		$A(D,d)=-B(D,d)$
	\end{theorem}
	\begin{proof}
		Exercise $5.3.1$ of \cite{DS} says that for $f\in M_k(\Gamma_0(N))$ and $n\in \ZZ^+$, we have $T_n(f)=\sum_{\gamma\in \tilde{M_n}}f|_k(\gamma)$ with $$S_n=\{d:d|n,gcd(n/d,N)=1\}$$ and $$\tilde{M_n}=\cup_{d\in S_n}\cup_{j=0}^{d-1}\begin{pmatrix}n/d&j\\0&d\end{pmatrix}$$ If $n=4p$, we get that $$\tilde{M_{4p}}=\bigcup_{j=0}^{4p-1}\begin{pmatrix}1&j\\0&4p\end{pmatrix}$$ Hence, if we let  $$fg|_{U(4p)}=\sum_{\ZZ_{4p}}fg(\frac{\tau+j}{4p})$$ for $f,g$ of weight $1/2,3/2$ respectively, we get $$\sum_{\ZZ_{4p}}fg(\frac{\tau+j}{4p})=\sum_{\ZZ_{4p}}(4p)fg|_2\begin{pmatrix}1&j\\0&4p\end{pmatrix}=4pT_{4p}(fg)$$ Also, by definition of Hecke operators (which is the same if we allow the pole at infinity), we have $$T_{4p}(fg)=\sum_{\mu\in \Gamma\backslash M_{4p}}fg|_2(\mu\tau)$$ and this function is $\Gamma-$invariant. It follows that $fg|_{U(4p)}$ is $SL_2(\ZZ)$ invariant and because it is a weakly holomorphic modular form of weight $2$, then it has to be a polynomial in $j'$. The constant term of $j'$ is 0 and we obtain $A(D,d)+B(D,d)=0$. 
	\end{proof}
	
	In order to compute a basis for each of these weights we explicit the previously seen algorithm. The idea is that we will construct a basis of weight $1/2$ forms in the Kohnen plus-space, then a basis of weight $3/2$ forms in the Kohnen plus-space using the structure theorem of Jacobi forms and the isomorphism between these two spaces. One can also be satisfied with the fact that knowing the $f_d's$ immediatly implies knowing the $g_D's$ by the Zagier duality but The algorithm goes as follows. We pick a $p|\#\MM$ and we compute all quadratic residues modulo $4p$. We consider the list $$\{f_d(\tau)=q^{-d}+\sum_{D>0,D\equiv \square[4p]}A(D,d)q^D : -d\equiv\square[4p]\}$$ These are the forms that we will use to construct all other $f_d's$ by induction. To do so, pick $d\geq 4p,-d\equiv\square[4p]$ not one of the quadratic residue modulo $4p$ of the above list and consider the associated form $f_{d-4p}$. Multiply $f_{d-4p}$ by $j(4p\tau)$ to get a form in the plus-space with leading coefficient $q^{-d}$. Substract linear combinations of the $f_{d'}$ with $0\leq d'<d$. This produces a new element of the basis. To construct explicitly the $f_d's$, we will use the Rankin-Cohen bracket. Call $f_0=\theta(\tau)$ the usual theta function and note that it is a form in $M_{1/2}(\Gamma_0(4p))^{!,+}$. From it we compute the forms $[f_0(\tau),E_{12-2n}(4p\tau)]_n/\Delta(4p\tau)$ and do linear combinations of all of them to compute the initial $f_d$'s (sometimes it will be necessary to compute Rankin-Cohen Bracket with $f_d's$ other than $f_0$). These Rankin-Cohen brackets are $$[\theta(\tau),E_{10}(4p\tau)]_1=\frac{1}{2}\theta(\tau)E'_{10}(4p\tau)=10\theta'(\tau)E_{10}(4p\tau)$$ $$[\theta(\tau),E_{8}(4p\tau)]_2=36\theta''(\tau)E_8(4p\tau)-\frac{27}{2}\theta'(\tau)E'_8(4p\tau)+\frac{3}{8}\theta(\tau)E''_8(4p\tau)$$ $$[\theta(\tau),E_{6}(4p\tau)]_3=-56\theta'''(\tau)E_6(4p\tau)+70\theta''(\tau)E'_6(4p\tau)-15\theta'(\tau)E_6''(4p\tau)+\frac{5}{16}\theta(\tau)E'''_6(4p\tau)$$ $$[\theta(\tau),E_{4}(4p\tau)]_4=35\theta''''(\tau)E_4(4p\tau)-\frac{245}{2}\theta'''(\tau)E'_4(4p\tau)+\frac{745}{8}\theta''(\tau)E''_4(4p\tau)-\frac{245}{16}\theta'(\tau)E_4'''+\frac{35}{128}\theta(\tau)E_4''''(4p\tau)$$
	For $p=2$ for example, we have that the quadratic residues modulo $8$ are $0,1,4$. Hence, we have that the initial $f_d's$ to construct are $f_0,f_4,f_7$. If $$u(\tau)=\frac{[\theta(\tau),E_{10}(4p\tau)]_1/\Delta(8\tau)+1056\theta(\tau)}{-20}$$ and $$v(\tau)=\frac{[\theta(\tau),E_{8}(4p\tau)]_2/\Delta(8\tau)-11520\theta(\tau)}{72}$$ then $$f_0(\tau)=\theta=1+2q+2q^4+...$$ $$f_4(\tau)=(v(\tau)-u(\tau))/12=q^{-4}-52q+272q^4+...$$ $$f_7(\tau)=(4u(\tau)-v(\tau))/3=q^{-7}-23q-2048q^4+...$$ To construct $f_8(\tau)$, it suffices to consider $f_0(\tau)$, multiply it by $j(8\tau)$ and then substract linear combinations of the initial $f_d's$, this gives $$f_8(\tau)=\theta(\tau)j(8\tau)-2f_4(\tau)-2f_7(\tau)-744\theta(\tau)=q^{-8}+152q+3552q^4+...$$
	Here, there are in fact two ways to obtain the basis in weight $3/2$. Either we use the $f_d's$ constructed above and the Zagier duality or we use the fact that $J_{2,p}(\Gamma^J)$ in bijection with $M_{3/2}^+(\Gamma_0(4p))$ under identification of Fourier coefficients. Thus, we can construct the desired Jacobi forms using the structure theorem, compute explicitly their Fourier expansion and use the coefficient identification to construct the $g_D's$ for $D$ being a quadratic residue modulo $4p$, strictly greater than 0. Once we have these initial forms $g_D,D\equiv\square[4p], 0<D\leq 4p$, the idea is to construct the other forms using the same algorithm as above. $J^!_{even,*}$ is a free polynomial algebra over $M^!(\Gamma)=\CC[E_4,E_6,\Delta^{-1}]$ with generators $a=\tilde{\phi}_{-2,1}$, $b=\tilde{\phi}_{0,1}$. We have seen that $\exists ! \phi_{D,p}\in J_{2,p}^!(\Gamma^J)$ with Fourier coefficients $B(D,4np-r^2)$ satisfying $B(D,-D)=1$, $B(D,d)=0$ for $d=4pn-r^2<0$ and $d\neq -D$, where we recall that $"!"$ denotes the Jacobi forms with $n$ and $r$ running over $\ZZ$ in the Fourier expansion. The structure theorem implies that $$\phi_{D,p}(\tau,z)=\sum_{k=0}^{p}f_{2k+2}a^kb^{p-k}$$ with $f_{2k+2}$ a form in $M^!_*(\Gamma)$ of weight $2k+2$. Consider $p=2$, then $D=1,4,8$. For $D=1$, we can see that the equation $8n-r^2=-1$ has solutions only for $n\geq 0$ so the above polynomial will take coefficient in $M_*(\Gamma)$. This gives $$\phi_{1,2}=c_1E_4ab+c_2E_6a^2$$ for some constants $c_1,c_2$. The same reasoning applies for $D=4$ and we get $$\phi_{4,2}=c_1E_4ab+c_2E_6a^2$$ with $$c_1E_4ab=c_1\Big((\zeta^{-2}+8\zeta^{-1}-18+8\zeta+\zeta^2)+(8\zeta^{-3}+144\zeta^{-2}+2232\zeta^{-1}-4768+2232\zeta+144\zeta^{2}+8\zeta^{3})q+...\Big)$$ $$c_2E_6a^2=c_2\Big((\zeta^{-2}-4\zeta^{-1}+6-4\zeta+\zeta^{2})+(-4\zeta^{-3}-480\zeta^{-2}+1956\zeta^{-1}-2944+1956\zeta-480\zeta^{2}-4\zeta^{3})q+...\Big)$$ For $D=1$, we know that $A(1,4)=-B(1,4)=-52$ and $A(1,7)=-B(1,7)=-23$. From here, we will solve the system of equation $$52=144c_1-480c_2,\,\,23=2232c_1+1956c_2$$ It gives $c_1=1/12=-c_2$ and $$\phi_{1,2}=(\zeta^{-1}-2+\zeta)+(\zeta^{-3}+52\zeta^{-2}+23\zeta^{-1}-152+23\zeta+52\zeta^{2}+\zeta^{3})q+...$$
	For $D=4$, we have $B(4,8n-r^2=4)=-272$ and $B(4,8n-r^2=7)=2048$. We can thus look at a coefficient with $n=1,r=-2$ and with $n=1,r=-1$ in $c_1E_4ab$ and $c_2E_6a^2$ respectively. It gives the system of equations $$144c_1-480c_2=-272,\,\,2232c_1+1956c_2=2048$$ with solutions $c_1=1/3=c_2/2$. We obtain $$\phi_{4,2}=(\zeta^{-2}-2+\zeta^{2})+(-272\zeta^{-2}+2048\zeta^{-1}-3552+2048\zeta-272\zeta^{2})q+...$$
	Now lets come back to where we left. Our goal is to prove that the generating function for the trace is a weight $3/2$ modular form and that we can provide infinite product expansions for the modified Hilbert class polynomial of our hauptmoduln $j_p$ for $p|\#\MM$. Zagier's theorem solves the first question for $p=2,3,5$. Lets now compute the desired product expansion for these primes. Following the notation of \cite{F}, we let $p^{(\tilde{p})}=p/(\tilde{p},p)=1,p$ depending on if $p=\tilde{p}$ or not and $S^{(\tilde{p})}$ be the set of hall divisors of $p$ that divides $p^{(\tilde{p})}$. If we let $t^{(\tilde{p})}$ be the Hauptmodul of the group generated by  $\Gamma_0(p^{(\tilde{p})})$ and all Atkin-Lehner involutions $W_{e,p^{(\tilde{p})}}$ with $e\in S^{(\tilde{p})}$, we get that $t^{(p)}=J(\tau)$ the hauptmodul of $SL_2(\ZZ)$. Otherwise, $t^{(\tilde{p})}=j_p(\tau)$ the hauptmodul of $\Gamma_0(p)^*$. Lets define the generalized trace $$Tr_{m}(d)=\sum_{Q\in \Gamma_0(p)^*\backslash Q_{d,p}}\frac{1}{w_Q}j_{p,m}(\alpha_Q)$$ where $j_{p,m}=j_p|_{T(m)}=\sum_{\gamma\in\Gamma\backslash M_m}j_p(\gamma\tau)=\sum_{ad=m,b[d]}j_p(\frac{a\tau+b}{d})=q^{-m}+O(q)$ is the unique such form. Also, we define $$\tilde{Tr}_m(d)=\sum_{Q\in \Gamma\backslash Q_d}\frac{1}{w_Q} J_m(\alpha_Q)$$ to be the original generalized trace function studied by Zagier.\\
	Let $m$ be a positive integer coprime to $p$ and consider $d=4pn-r^2$. 
	\begin{lemma}
		In the above setting, $$Tr_m(d)=-\text{ coefficient of }q^n\zeta^r\text{ in }\phi_{1,p}|_{T_m}$$ with $T_m$ being the Hecke operator acting on Jacobi forms or equivalently on half-integral forms.
	\end{lemma}
	\begin{proof}
		We will argue by induction. We first prove the result for $q$ a prime divisor of $m$, then we prove it for a power of $q$ being a Hall divisor of $m$ and then for $m=lq^s$ with $(l,q)=1$.\\
		Let $q$ be a prime divisor of $m$ and not the usual $e^{2\pi i \tau}$. Then $$Tr_{q}(d)=\sum \frac{1}{w_Q}j_{p,q}=\sum \frac{1}{w_Q}j_p|_{T(q)}(\alpha_Q)=Tr_1(dq^2)+\Big(\frac{-d}{q}\Big)Tr_1(d)+qTr_1(d/q^2)$$ with $Tr_1(d/q^2)=0$ if $d/q^2\notin \ZZ$ by theorem 5.ii of \cite{Za}. Using theorem 8 of \cite{Za}, we get $$=-\Big(B(dq^2)+\Big(\frac{-d}{q}\Big)B(d)+pB(d/q^2)\Big)=-B_q(d)$$ where $B_q(d)$ is the coefficient of $q^d$ in $g_{1,p}|_{T_q}$ (which is the same as the coefficient of $q^n\zeta^r$ in $\phi_{1,p}|_{T_q}$) by theorem 4.5 of \cite{ZaEi}.\\
		Now let $q^s||m$ be a Hall divisor of $m$. Then by properties of the Hecke operator, we have
		\begin{claim}
			$$j_p|_{T(q^s)}=j_{p,q^{s-1}}|_{T(q)}-qj_{p,q^{s-2}}$$
		\end{claim}
		\begin{proof}{(claim)\\}
			Let $t=j_p, t_m=j_p,m$ and replace $q$ by $u$ in the below proof. Then $$t|_{T(u^s)}=t_{u^s}=q^{-u^s}+O(q)$$ and $$t_{u^{s-1}}|_{T(u)}=\sum_{ad=u,0\leq b<d}t_{u^{s-1}}^{(a)}(\frac{a\tau+b}{d})$$ but since $u$ is prime, we get $1\cdot u=u\cdot 1$ and so $$=\sum_{0\leq b<u}t_{u^{s-1}}(\frac{\tau+b}{u})+t_{u^{s-1}}^{(u)}(u\tau)=\sum_{0\leq b<u}t_{u^{s-1}}(\frac{\tau+b}{u})+t_{u^{s-1}}(u\tau)$$ The term in the summation has leading coefficient being $q^{-u^{s-2}}$ and after taking the sumamtion we get $u$ of these terms. The second term on the right hand side has leading coefficient $q^{-u^s}$. This gives for the right hand side after cancellation of terms $$q^{-u^s}+O(q)$$ but such a form is unique and must be $t|_{T(u^s)}$.
		\end{proof}
		This gives us that $$Tr_{q^s}(d)=Tr_{q^{s-1}}(dq^2)+\Big(\frac{-d}{q}\Big)Tr_{q^{s-1}}(d)+qTr_{q^{s-1}}(d/q^2)-qTr_{q^{s-2}}(d)$$ and suppose by induction on $s$ that the above equals $$=-\Big(B_{q^{s-1}}(dq^2)+\Big(\frac{-d}{q}\Big)B_{q^{s-1}}(d)+qB_{q^{s-1}}(d/p^2)\Big)+qB_{q^{s-2}}(d)$$ then by the isomorphism $$B_{q^{s-1}}=\text{coeff of }g_1|_{T_{u^{s-1}}}=\phi_1|_{T_{q^{s-1}}}$$ we obtain $$=-\text{coefficient of }q^n\zeta^r \text{ in }\phi_{1,p}|_{T_{q^{s-1}}\circ T_{q}}-q\phi_{1,p}|_{T_q^{s-2}}$$ $$=-\text{coefficient of }q^n\zeta^r \text{ in }\phi_{1,p}|_{T_{q^{s}}}$$ by Corollary 1 of \cite{F} and the last claim.\\
		Finally, let $m=lq^s$ with $(l,q)=1$. Let $n(m)$ denotes the number of factor of $m$. We use induction on $n(m)$. If $n(m)=1$ then this reduces to the previous case and the base case is done. Now for $$j_p|_{T(m)}=j_p|_{T(l)\circ T(q^s)}$$ by coprimality of $q^s,l$ and as seen previously, we get that $$=j_{p,l}|_{T(q^s)}=j_{p,lq^{s-1}}|_{T(q)}-qj_{p,lq^{s-2}}$$ By taking the traces we get $$Tr_m(d)=Tr_{lq^{s-1}}(dp^2)+\Big(\frac{-d}{q}\Big)Tr_{lq^{s-1}}(d)+qTr_{lq^{s-1}}(d/q^2)-qTr_{lq^{s-2}}(d)$$ $$=-\text{coefficient of }q^n\zeta^r \text{ in }\phi_{1,p}|_{T_{lq^{s-1}}\circ T_q}-q\phi_{1,p}|_{T_{lq^s}}$$ $$=-\text{coefficient of }q^n\zeta^r \text{ in }\phi_{1,p}|_{T_{lq^s}}$$ by induction on $s$.
	\end{proof}
	\begin{lemma}
		Let $m$ as in the above lemma. Then $$\phi_{1,p}|_{T_m}=\sum_{v|m}v\phi_{v^2,p}$$ 
	\end{lemma}
	\begin{proof}
		Let $q$ a prime dividing $m$, $q^s|| m$ be a Hall divisor of $m$ and $\phi_1=\phi_{1,p}$. First we show that $$\phi_1|_{T_{q^s}}=\sum_{i=0}^sq^i\phi_{q^{2i}}$$ For $s=1$, then the coefficient of $g_1|_{T_q}$ is $$B(dq^2)+(\frac{-d}{q})B(d)+qB(d/q^2)=1,q,0$$ whether $d=-1,-p^2$ or $<0$ different of $-1$ and $-p^2$. To see this, for $d=-1$, note that $b(-q^2)=0$ and that $B(-1)=1$ which implies that $(\frac{1}{q})B(-1)=1$ since $1$ is a quadratic residue modulo $q$. For $d=-p^2$, we get $B(-p^4)=0$, $(\frac{p^2}{p})B(-p^2)=0$ and $pB(-1)=p$. For the last case, this follows from the definition of $g_{1,p}$.\\
		Thus we can write $$g_1|_{T_q}=pg_{p^2,p}+g_{1,p}$$ $$=p\Big(q^{-p^2}+\sum_{d\geq 0}B(q^2,d)q^2\Big)+q^{-1}+\sum_{d\geq 0}B(1,d)q^d$$ $$pq^{-p^2}+q^{-1}+\sum_{d\geq 0}\Big(pB(p^2,d)+B(1,d)\Big)q^d$$ and therefore by uniqueness of the forms $q^{-D}+c+O(q)$, we get $$\phi_{1}|_{T_p}=p\phi_{p^2}+\phi_1$$
		Now let $s\geq 2$. Then the coefficient of $q^d$ in $g_1|_{T_{q^s}}$ is $$B_{q^{s-1}}(dq^2)+(\frac{-d}{q})B_{q^{s-1}}(d)+qB(d/q^2)-qB_{q^{s-2}}(d)$$ by claim 2 and what follows about the trace. This implies that $$g_1|{T_{q^s}}=g_1|_{T_{q^{s-1}}\cdot T_q}-qg_1|_{T_{q^{s-2}}}$$ $$\iff$$ $$\phi_1|_{T_{q^{s}}}=\phi_1|_{T_{q^{s-1}}\cdot T_q}-q\phi_1|_{T_{q^{s-2}}}$$ which gives by the induction hypothesis $$(\sum_{i=0}^{s-1}q^i\phi_{q^{2i}})|_{T_q}-q\sum_{i=0}^{s-2}q^i\phi_{q^{2i}}$$ For a fixed $i>0$, the coefficient of $q^d$ in $g_{q^{2i}}|_{T_q}$ (under the correspondence) is $$B(q^{2i},dq^2)+(\frac{-d}{q})B(q^{2i},d)+qB(q^{2i},d/q^2)=1,p,0$$ whenever $d=-q^{2i-2},-q^{2i+2},d<0\,\&\,\neq -q^{2i\pm2}$. This shows that $\phi_{q^{2i}}|_{T_q}=\phi_{q^{2i-2}}+q\phi_{q^{2i+2}},\phi_1+q\phi_{q^2}$ for $i>0$ and $i=0$ respectively. Hence we get $$\phi|_{T_q^s}=(\sum_{i=0}^{s-1}q^i\phi_{q^{2i}})|_{T_q}-q\sum_{i=0}^{s-2}q^i\phi_{q^{2i}}$$ $$=\sum_{i=1}^{s-1}q^i(\phi_{q^{2i-2}}+q\phi_{q^{2i+2}})+\phi_1+q\phi_{q^2}-q\sum_{i=0}^{s-2}q^i\phi_{q^{2i}}=\sum_{i=0}^sq^i\phi_{q^{2i}}$$ Now for $m=lq^s,(l,q)=1,$ we get, as in the last lemma using induction on $n(m)$ that $$\phi_1|_{T_m}=\phi_{T_l\cdot T_{q^s}}=(\sum_{v|l}v\phi_{v^2})|_{T_{q^s}}=\sum_{v|m}v\phi_{v^2}$$
	\end{proof}
	\begin{theorem}
		Let $d=4pn-r^2$, $p=2,3,5$. Then $$Tr_m(d)=-\text{coefficient of }q^n\zeta^r \text{ in }\sum_{0<u|m}2^\xi u\phi_{u^2,p}$$ with $\xi=0$ if $u\neq kp$ and $1$ if $u=kp$ for some $k\in \NN$.
	\end{theorem}
	\begin{proof}
		
		Let $p$ as in the theorem. By proposition 2.6 of  \cite{F}, we have that the generalized Hecke operator $T(p)$ satifies $$T(p^k)\circ T(p)=T(p^{k+1})+pR_p\circ T(p^{k-1})$$ with $j_{p,n}|_{R_p}=J_n$, the Hauptmodul of $SL_2(\ZZ)$ and $k\geq 0$. For $(l,p)=1$, we obtain $$j_{p,lp^{k+1}}=j_{p,l}|_{T(p^{k+1})}=j_{p,l}|_{T(p^{k})\circ T(p)}-pJ_l|_{T(p^{k-1})}$$ $$=j_{p,lp^k}|_{T(p)}-pJ_{lp^{k-1}}=J_{lp^k}(p\tau)+pj_{p,lp^k}|_{U_p}-pJ_{lp^{k-1}}$$ Using theorem 3.7 of \cite{F} gives the formula $$pj_{p,lp^k}|_{U_p}+j_{p,lp^k}=J_{lp^k}+pJ_{lp^{k-1}}$$ Combining both formulas gives $$j_{p,lp^{k+1}}(\tau)=J_{lp^k}(p\tau)+J_{lp^k}(\tau)-j_{p,lp^k}(\tau)$$ and hence 
		$$\sum_{Q\in \Gamma_0(p)^*\backslash Q_{d,p}}j_{p,lp^{k+1}}(\alpha_Q)=\sum_{Q\in \Gamma_0(p)^*\backslash Q_{d,p}}(J_{lp^k}(p\tau)+J_{lp^k}(\tau))|_{\tau=\alpha_Q}-\sum_{Q\in \Gamma_0(p)^*\backslash Q_{d,p}}j_{p,lp^k}(\alpha_Q)$$ Recall the bijection between $\Gamma_0(p)^*\backslash Q_{d,p}$ and $SL_2(\ZZ)\backslash Q_d$. It gives $$\sum_{Q\in \Gamma_0(p)^*\backslash Q_{d,p}}j_{p,lp^{k+1}}(\alpha_Q)=2\sum_{Q\in SL_2(\ZZ)\backslash Q_d}J_{lp^k}(\alpha_Q)-\sum_{Q\in \Gamma_0(p)^*\backslash Q_{d,p}}j_{p,lp^k}(\alpha_Q)$$ Taking now traces on both sides gives $$Tr_{lp^{k+1}}(d)=2\tilde{Tr}_{lp^k}(d)-Tr_{lp^k}(d)$$ for $k\geq 0$.\\
		Let $m=lp^k,(l,p)=1$. We use induction on $k$. For $k=0$ we have already proved the result in the last lemma. Suppose true for $k$ now, then $$Tr_{lp^{k+1}}(d)=2\tilde{Tr}_{lp^k}(d)-Tr_{lp^k}(d)$$ as showed above. We know that $$\tilde{Tr}_{lp^k}(d)=-\text{coeff of }q^d\text{ in }g_{1,1}|_{T_{lp^k}}\in M_{3/2}^{!,+}(\Gamma_0(4))$$ with corresponding Jacobi form $\phi_{1,1}|_{T_{lp^k}}\in J_{2,1}^{!}$ which implies $$2\tilde{Tr}_{lp^k}(d)=-2\text{ coeff of }q^d\text{ in }g_{1,1}|_{T_{lp^k}}=-2\text{ coeff of }q^n\zeta^r\text{ in }\phi_{1,1}|_{T_{lp^k}}$$ and thus $$Tr_{lp^{k+1}}(d)=2\tilde{Tr}_{lp^k}(d)-Tr_{lp^k}(d)$$ $$=-\text{ coeff of }q^n\zeta^r\text{ in }(2\phi_{1,1}|_{T_{lp^k}\circ V_p}-\sum_{u|lp^k}2^\xi u\phi_{u^2})$$ by the induction hypothesis. $$=-\text{ coeff of }q^n\zeta^r\text{ in }(2\phi_{1,1}|_{T_{lp^k}\circ V_p}-(\sum_{i=1}^k\sum_{v|l}2vp^i\phi_{v^2p^{2i},p}+\sum_{v|l}v\phi_{v^2,p}))$$ $$=-\text{ coeff of }q^n\zeta^r\text{ in }(2\sum_{i=0}^k\sum_{v|l}vp^i\phi_{v^2p^{2i},1}|_{V_p}-(\sum_{i=1}^k\sum_{v|l}2vp^i\phi_{v^2p^{2i},p}+\sum_{v|l}v\phi_{v^2,p}))$$
		\begin{lemma}
			$$\phi_{v^2p^{2i},1}|_{V_p}=p\phi_{v^2p^{2i+2},p}+\phi_{v^2p^{2i},p}$$ for $(v,p)=1$ where $|_{V_p}:J_{2,1}^!\to J_{2,p}^!, \, \sum_{n,r}c(n,r)q^n\zeta^r\mapsto \sum_{n,r}(\sum_{a|(n,r,p)}ac(np/a^2,r/a))q^n\zeta^r$.
		\end{lemma}
		\begin{proof}{(lemma)\\}
			For $c(n,r)=c(4n-r^2)=B(v^2p^{2i},4n-r^2)$, we get $B(v^2p^{2i},4n-r^2)\mapsto \sum_{a|(n,r,p)}ac(4np/a^2-r^2/a^2)=\sum_{a|(n,r,p)}ac(4np-r^2/a^2)$. If $4np-r^2=-v^2p^{2i+2}$, then $$=c(-v^2p^{2i+2})+pc(-v^2p^{2i})=B(v^2p^{2i},-v^2p^{2i+2})+pB(v^2p^{2i},-v^2p^{2i})=0+p=p$$ If $4np-r^2=-v^2p^{2i},$ then $$=1+p$$ by exactly the same approach and finally, it is $0$ otherwise. Hence $\phi_{v^2p^{2i},1}|_{V_p}$ is really of the claimed form by uniqueness of $\phi_{D,p}$.
		\end{proof}
		It follows by the above lemma that $$Tr_{lp^{k+1}}(d)=-\text{coefficient of }q^n\zeta^r\text{ in } (2\sum_{i=0}^k\sum_{v|l}vp^i(p\phi_{v^2p^{2i+2},p}+\phi_{v^2p^{2i},p})-\sum_{i=1}^k\sum_{v|l}2vp^i\phi_{v^2p^{2i},p}-\sum_{v|l}v\phi_{v^2})$$ From now we will focus only on the coefficient of $q^n\zeta^r$ $$=\sum_{i=1}^k\sum_{v|l}2vp^{i+1}\phi_{v^2p^{2i+2}}+2vp^i\phi_{v^2p^{2i}}+\sum_{v|l}2vp\phi_{v^2p^2}+v\phi_{v^2}-\sum_{i=1}^k\sum_{v|l}2vp^i\phi_{v^2p^{2i}}$$ $$=\sum_{i=1}\sum_{v|l}2vp^{i+1}\phi_{v^2p^{2i+2}}+\sum_{v|l}2vp\phi_{v^2p^2}+\sum_{v|l}v\phi_{v^2}$$ $$=\sum_{i=0}^k\sum_{v|l}2vp^{i+1}\phi_{v^2p^{2(i+1)}}+\sum_{v|l}v\phi_{v^2}=\sum_{i=1}^{k+1}2vp^i\phi_{v^2p^{2i}}+\sum_{v|l}v\phi_{v^2}$$
	\end{proof}
	Let $z\in \HH$ and consider $\frac{1}{m}j_{p,m}(z)$. Then $$j_p(\tau)-j_p(z)=q^{-1}\exp{(-\sum_{m=1}^\infty\frac{1}{m}j_{p,m}(z)q^m)}$$ Now we have all the tools to prove a generalized Borcherds product formula for the Fricke Group.\\
	We have $$j_p(\tau)-j_p(z)=q^{-1}\exp{(-\sum_{m=1}^{\infty}j_{p,m}(z)q^m/m)}$$ and by taking the modified Hilbert class polynomial on both sides we get $$\mathcal{H}_d(j_p(\tau))=\prod_{Q\in \Gamma_0(p)^*\backslash Q_{d,p}}(j_p(\tau)-j_p(\alpha_Q))^{1/w_Q}=q^{-H(d)}\exp{(-\sum_{m=1}^\infty\sum_{Q}\frac{j_{p,m}(\alpha_Q)}{w_Q}\frac{q^m}{m})}$$ $$=q^{-H(d)}\exp{(\sum_{m=1}^\infty Tr_m(d)\frac{q^m}{m})}$$
	Using our claim that $Tr_m(d)=-\sum_{u|m}u2^\xi B(u^2,d)$, we get $$\mathcal{H}_d(j_p(\tau))=q^{-H(d)}\exp{(\sum_{m=1}^\infty\sum_{u|m}u2^\xi B(u^2,d)q^m/m)}$$ $$=q^{-H(d)}\exp{(\sum_{u=1}^{\infty}u2^\xi B(u^2,d)\sum_{m=1}^\infty q^{um}/um)}$$ $$=q^{-H(d)}\exp{(-\sum_{u=1}^\infty 2^\xi B(u^2,d) \sum_{m=1}^\infty -(q^u)^m/m)}$$ $$=q^{-H(d)}\exp{(\sum_{u=1}^\infty -2^\xi B(u^2,d) log(1-q^u))}$$ $$=q^{-H(d)}\prod_{u=1}^\infty (1-q^u)^{-2^\xi B(u^2,d)}$$ $$=q^{-H(d)}\prod_{u=1}^\infty (1-q^u)^{2^\xi A(u^2,d)}$$ with $A(u^2,d)$ the coefficient of the associated dual form of weight $1/2$. 
	\pagebreak
	
	Now we prove the above product expansion for all primes $p|\#\MM$. The idea is that we will translate Borcherds theorem on automorphic forms with singularities on Grassmannian in terms of modular forms for the Fricke group. For some $p|\#\MM$, we will lift a multiple of the form $f_d$ to a vector valued modular form of weight $1/2$ for some Weil representation of a lattice $M$, denoted $\rho_M$. We then use Borcherds theorem as a black box and prove the existence of a meromorphic function of weight 0 on $\HH$ whose divisors are those of the modified Hilbert class polynomial of discriminant $-d$, $\mathcal{H}_{d,p}(j_p(\tau))$. As a corollary, we derive that the trace function is a Jacobi form of weight 1 and index p. It is thus a weight $3/2$ modular form on the Kohnen plus-space via identification of Fourier coefficients. For the rest of this section, we always consider the principal branch of the square root, namely $(-\pi/2,\pi/2]$.
	\begin{definition}
		The Metaplectic group $M_{p_2}(\RR)$ is the group of pairs $$(A,\phi(\tau))$$ with $M\in SL_2(\RR)$ and $\phi:\HH\to \CC$ a holomorphic function satisfying $\phi(\tau)^2=c\tau+d$. The group law is $$(A_1,\phi_1(\tau))(A_2,\phi_2(\tau))=(A_1A_2,\phi_1(A_2\tau)\phi_2(\tau))$$ which is well defined by the cocycle property of $j(A,\tau)=c\tau+d$.
	\end{definition}
	This group is a double cover of $SL_2(\RR)$. There is thus a locally isomorphic embedding from $SL_2(\RR)\hookrightarrow M_{p_2}(\RR)$ given by $A\mapsto (A,\sqrt{c\tau+d})$. we will refer to this map as the covering map. The idea of double cover can be understood from the existence of the short exact sequence $$\{\pm 1\}\hookrightarrow M_{p_2}(\RR)\twoheadrightarrow SL_2(\RR)$$ Let $M_{p_2}(\ZZ)$ be the image of $SL_2(\ZZ)$ in $M_{p_2}(\RR)$ under the covering map. Then $M_{p_2}(\ZZ)=\Big<(T,1),(S,\sqrt{\tau})\Big>$ with relations $S^2=(ST)^3\triangleq Z=(-I,i)$. We leave as a fact that $Z$ generates the center of $M_{p_2}(\ZZ)$. We now define a representation of this group arising from a lattice $M$ with attached quadratic form $Q(x)$. Recall that a lattice $M$ is a free $\ZZ$-module with positive definite bilinear form $B(x,y)$. From this bilinear form, we can define the quadratic form $Q(x)=\frac{1}{2}B(x,x)$. Hence, given a lattice $M$ and the (assumed) non degenerate attached quadratic form $Q(x)$, we can always diagonalize $Q$ so that $$Q(x)=\sum_{i=1}^{m}a_ix^i-\sum_{i=m+1}^{m+n}b_ix^i$$ with $Q$ acting on $M\otimes_{\ZZ} \QQ$ and where $a_i,b_i>0$. Using this diagonalization, we can define the signature of a lattice to be $(b^+,b^-)$ where $b^+=m,b^-=n$ are the number of positive and negative eigenvalues of the matrix form of $Q$. A natural object coming along with our lattice $M$ is the notion of dual lattice $M'=\{x\in M\otimes \QQ:B(x,y)\in\ZZ,\forall y\in M\}$ and discriminant group $M'/M$. As $M$ is a subgroup of same rank as $M'$, their quotient is a finite abelian group. This means that the associated group algebra $\CC[M'/M]$ is a finite dimensional vector space. This is the vector space that we will use to define the Weil representation of our Metaplectic group of order 2. From now on, we will denote the basis of $\CC[M'/M]$ by $(e_\gamma)_{\gamma\in M'/M}$. From there, to define a representation of $M_{p_2}(\ZZ)$ on $\CC[M'/M]$, it suffices to define how the generators $T,S$ act on the basis $(e_\gamma)_\gamma$. 
	\begin{definition}
		We define the Weil representation $$\rho_M: M_{p_2}(\ZZ)\to GL(\CC[M'/M])$$ by $$\rho_M(T)e_\gamma=e(Q(\gamma))e_\gamma$$ and $$\rho_M(S)e_\gamma=\frac{\sqrt{i}^{b^--b^+}}{\sqrt{|M'/M|}}\sum_{\delta\in M'/M}e(-B(\gamma,\delta))e_\delta$$ This representation is unitary.
	\end{definition}
	\begin{remark}
		This representation is often refered to in the litterature as being the Weil representation attached to the quadratic module $(M'/M,Q)$.
	\end{remark}
	Finally, one last definition 
	\begin{definition}
		Let $M$ be an even lattice of signature $(2,n)$ equipped with a non degenerate quadratic form $Q(x)=\frac{1}{2}B(x,x)$, $M'$ its dual and $\CC[M'/M]$ the group ring of $M'/M$ with basis $(e_\gamma)_{\gamma\in M'/M}$. A vector valued modular form of weight $k\in \frac{1}{2}\ZZ$ and type $\rho_M$ is a holomorphic function $F:\HH\to \CC[M'/M]$ $$F(\tau)=\sum_{\gamma\in M'/M}f_\gamma(\tau) e_\gamma$$ on the upper half-plane satisfying $$f_\gamma(\tau+1)=e(Q(\gamma))f_\gamma(\tau)$$ $$f_\gamma(-1/\tau)=\sqrt{\tau}^{2k}\frac{\sqrt{i}^{n-2}}{\sqrt{|M'/M|}}\sum_{\delta\in M'/M}e(-B(\gamma,\delta))f_\delta(\tau)$$ Remark that if we define $$F|_k(A)(\tau)=\sqrt{c\tau+d}^{-2k}\rho_M(A)^{-1}f(A\tau)$$ for $A\in M_{p_2}(\ZZ)$ we obtain that $F$ transforms like a modular form of weight $k$ and automorphy factor $\rho_M$ if $F|_k=F$. We can check that $F|_k(A)|_k(B)=F|_k(AB)$ and so it suffices to define a vector valued modular form only with respect to the generators $T,S$.
	\end{definition}
	Let $p$ be a prime dividing the order of the Monster $\MM$ and consider $M$ to be the $3$-dimensional even lattice of all symmetric matrices of the form $$\xi=\begin{pmatrix}c/p&-b/2p\\-b/2p&a/p\end{pmatrix}$$ with $a/p,b/2p,c\in \ZZ$ with norm $B(\xi,\xi)=-2pdet(\xi)=(b^2-4ac)/2p$. The lattice $M$ splits naturally according to its diagonal elements as the direct sum of the $2$-dimensional hyperbolic unimodular even lattice $\ZZ\sm 1/p&0\\ 0&0 \esm+\ZZ\sm 0&0\\ 0&1 \esm$ and the lattice $\ZZ\sm 0&1\\ 1&0 \esm$ of norm $2p$. Its dual lattice is the set of matrices $$\begin{pmatrix}c/p&-b/2p\\-b/2p&a/p\end{pmatrix}$$ with $a/p,b,c\in \ZZ$. The discriminant group $M'/M$ can be naturally identified with $\ZZ/2p\ZZ$ by mapping a matrix of $M'$ to the value of $b[2p]$ and $\Gamma_0(p)^*$ acts on $M$ via orthogonal transformations $v\mapsto XvX^t$ for $v\in M,X\in \Gamma_0(p)^*$. In fact the automorphism group of $M$ is precisely $\Gamma_0(p)^*$ \cite{Bo2}. This is the major fact we need in order to use Borcherds theorem. Consider $$Gr(M)=\{v\in M\otimes \RR:dim(v)=2,Q|_v \geq 0\}$$ the Grassmannian of $M$. If we let $X_M$ and $Y_M$ (of same norm) be an oriented orthogonal base of some element $v\in Gr(M)$, then we can map $v$ to the point $Z_M=X_M+iY_M\in M\otimes \CC$ (which implies that $Z_M$ has norm 0 by definition of $X_M,Y_M$). We can see that any such basis (and thus point $Z_M$) is represented by a matrix $\sm \tau^2&\tau\\ \tau&1 \esm$ for $\tau\in \HH$. This is simply mapping a point $\tau\in \HH$ to the $2$-dimensional positive definite space spanned by the real and imaginary part of the norm 0 vector $\sm \tau^2&\tau\\ \tau&1 \esm$ which is indeed a surjective application and every element of the Grassmannian can be reached this way. Now we recall some definitions for modular forms of half-integral weight. Let $\mathfrak{G}$ be the group consisting of all pairs
	$(A,\psi (\tau))$ with $A=\sm a&b\\c&d \esm\in
	GL_2^+(\mathbb{R})$ and $\psi (\tau)$ a holomorphic function on $\HH$ satisfying $|\psi (\tau)|=(\det
	A)^{-1/4} |c\tau +d|^{1/2}$ and whose group operation is defined by $$(A,\psi_1 (\tau))(B,\psi_2 (\tau))=(AB, \psi_1 (B\tau) \psi_2 (\tau))$$
	For $f:\HH \to \mathbb{C}$ and $\xi=(A,\psi (\tau))\in
	\mathfrak{G}$, we let $f|_{[\xi]_{k+1/2}}=f|_\xi=\psi (\tau)^{-2k-1} f(A\tau)$. It follows that
	$f|_{\xi_1}|_{\xi_2}=f|_{\xi_1 \xi_2}$. There is a natural monomorphism
	$\Gamma_0(4)\to \mathfrak{G}$ given by $A\mapsto
	A^*=(A,j(A,\tau))$, where $j(A,\tau)=\left(\frac{c}{d} \right)
	\left(\frac{-1}{d} \right)^{-1/2} (c\tau +d)^{1/2}$
	for $A=\sm a&b\\c&d \esm$.
	Then, we'll say that $f$ is a modular form of weight $k+1/2$ on
	$\Gamma_0(4N)$ if it satisfies $f|_{A^*}=f$ for every $A\in
	\Gamma_0(4N)$ and is holomorphic at the cusps. Otherwise, if we allow our function to have poles at cusps, we'll employ the natural terminology of saying that our function is a weakly holomorphic modualr form of half-integral weight.\\
	Let $f=\sum_{n\in\mathbb{Z}} c(n)q^n \in M_{1/2}^! (\Gamma_0(4p))$ be a weakly holomorphic modular form of weight $1/2$ and consider $\beta \in\mathbb{Z}/2p\mathbb{Z}$ for which we set $$ h_\beta (\tau) = 2^{s(\beta,p)}
	\sum_{n\equiv \beta^2 [4p]} c(n) q^{n/4p} $$ where $s(\beta,p)=1$ if $\beta\equiv 0,p\,\,[2p]$ and 0 otherwise. We can see that $$h_\beta (\tau+1) = \zeta_{4p}^{\beta^2} h_\beta (\tau)$$ Also, for $j$ coprime to $4p$, we let $f_j= f|_{\left(\sm 1&j\\0&4p\esm,(4p)^{1/4}\right)}|_{W_{4p}}$ where $W_{4p}=\left( \sm
	0&-1\\4p&0 \esm,(4p)^{1/4}\sqrt{-i\tau} \right)$. By definition
	$f_j=f|_{\left(\sm 4pj&-1\\16p^2&0 \esm, 2\sqrt{-pi\tau}\right)}$ and if we let $b,d$ integers satisfying $jd-4pb=1$ with $\psi_j=\left(\frac{4p}{j}\right)\sqrt{\left(\frac{-1}{j}\right)}
	\zeta_8^{-1}$, then we get that $\sm j&b\\4p&d \esm^* \left( \sm
	4p&-d \\ 0&4p \esm, \psi_j \right) =\left(\sm 4pj&-1\\16p^2&0
	\esm, 2\sqrt{-pi\tau}\right)$. Hence we can rewrite $$f_j=f|_{\sm j&b\\4p&d \esm^*
		\left( \sm 4p&-d \\ 0&4p \esm, \psi_j \right)}=\psi_j^{-1}
	f\left(\tau-\frac{j^{-1}}{4p}\right)$$ and see that
	\begin{equation}
	f|_{\left(\sm 1&j\\0&4p \esm,
		(4p)^{1/4}\right)}|_{W_{4p}}=\psi_j^{-1}
	f\left(\tau-\frac{j^{-1}}{4p}\right) \label{AA}
	\end{equation}
	where $j^{-1}$ denotes the inverse of $j$ modulo $4p$. The left hand side of (\ref{AA}) is
	\begin{align}
	(4p)^{-1/4} f\left(\frac{\tau+j}{4p}\right)
	|_{W_{4p}} &=(4p)^{-1/4}(\sum_{n\in\mathbb{Z}}
	c(n)\zeta_{4p}^{nj}q^{n/4p})|_{W_{4p}}
	= (4p)^{-1/4}\frac{1}{2}\left( \sum_{\beta [2p]} \zeta_{4p}^{\beta^2j}h_\beta
	\right)|_{W_{4p}} \notag \\
	&=(4p)^{-1/2}\frac{1}{2}(\sqrt{-i\tau})^{-1}
	\left( \sum_{\beta [2p]} \zeta_{4p}^{\beta^2j}h_\beta
	\right)\left(-\frac{1}{4p\tau} \right). \label{BB}
	\end{align}
	and the right hand side is
	\begin{equation}
	\psi_j^{-1}(\sum_{n\in\mathbb{Z}} c(n)\zeta_{4p}^{-nj^{-1}}q^{n})
	=\psi_j^{-1}\frac{1}{2}\left(\sum_{\beta [2p]} \zeta_{4p}^{-\beta^2j^{-1}}h_\beta
	\right)(4p\tau).
	\label{CC}
	\end{equation}
	Replacing $\tau$ by $\tau/4p$ in (\ref{BB}) and (\ref{CC}), we get \begin{equation}\label{DD}
	\sum_{\beta [2p]} \zeta_{4p}^{\beta^2j}h_\beta(-1/\tau)=\left(
	\frac{4p}{j}\right)\sqrt{\left(\frac{-1}{j}\right)}^{-1}\sqrt{\tau}
	\cdot \sum_{\beta [2p]}
	\zeta_{4p}^{-\beta^2j^{-1}}h_\beta(\tau).
	\end{equation}
	Let $R$ be a $2p\times 2p$ matrix defined by
	$R=\frac{\zeta_8^{-1}}{\sqrt{2p}}\left(
	\zeta_{2p}^{-l\beta}\right)_{0\le l,\beta < 2p}.$ To show that
	$\sum_{\beta [2p]} h_{\beta} \mathfrak{e}_{\beta} $ is a vector
	valued modular form, we check that
	\begin{equation} \label{inversion}
	\sm h_0 \\ \vdots \\ h_\beta \\ \vdots\esm \left( -1/\tau\right)
	=\sqrt{\tau} \, R\sm h_0 \\ \vdots \\ h_\beta \\ \vdots\esm
	(\tau).
	\end{equation}
	Since $h_\beta=h_{-\beta}$, the above identity is equivalent to
	\begin{equation} \label{EE}
	A\sm h_0 \\ \vdots \\ h_\beta \\ \vdots\esm \left( -1/\tau\right)
	=\sqrt{\tau} \, AR\sm h_0 \\ \vdots \\ h_\beta \\ \vdots\esm
	(\tau)
	\end{equation}
	for some matrix $A$ with $2p$ columns of which the first $p+1$
	ones are linearly independent and span the other ones. We take
	$A=\left(
	\zeta_{4p}^{\beta^2 j_l}\right)_{1\le l \le \varphi (4p) \atop
		0\le \beta < 2p}$ where $j_l$ is the $l$-th largest element in the
	set $\{j \, | \, 1\le j \le 4p \text{ and } (j,4p)=1 \}$. We can
	check that the rank of $A$ is $p+1$ unless $p=2$ but since we've already worked out this case before, we will not go over it. From
	the Gauss quadratic sum identity it follows that
	$AR=\left( \left(
	\frac{4p}{j_l}\right)\sqrt{\left(\frac{-1}{j_l}\right)}^{-1}
	\zeta_{4p}^{-\beta^2j_l^{-1}} \right)_{1\le l \le \varphi (4p) \atop
		0\le \beta < 2p}$. Identity (\ref{EE})
	follows from (\ref{DD}).\\
	Now we work out the Borcherds lift of our vector valued modular form. Take $M$ as before, let $G(M)$ denote the Grassmannian of $M$ and recall that it can be identified with the upper half plane by mapping $\tau\in \h$ to the two dimensional positive definite spaces spanned by the real and imaginary parts of the norm zero vector $\sm \tau^2&\tau\\ \tau&1 \esm$. Let $F=\sum_\beta h_\beta \mathfrak{e}_\beta$ be the vector valued
	modular form associated to
	$f=\sum_{n\in\mathbb{Z}}c(n)q^n\in M_{1/2}^! (\Gamma_0(4p))$ as presented above. We write $\sum_{n\in\mathbb{Z}} c_\beta(n/4p)e(n\tau/4p)$
	for the Fourier expansion of $h_\beta(\tau)$ so that $c_\beta(n/4p)=2^{s(\beta,p)}c(n)$ for $n\equiv \beta^2[4p]$ and 0 otherwise. Borcherds theorem \cite{Bo3} implies that there is a meromorphic function $\Psi_f$ on $G(M)$ that has an infinite product expansion (explicitly given by Item 5 in Theorem 13.3 of \cite{Bo3}) satisfying
	\begin{itemize}
		\item[1.]$\Psi_f$ is an automorphic form of weight $c_0(0)/2$ for the group $Aut(M,F)$ with respect to some unitary character $\chi$ of $Aut(M,F)$.
		\item[2.] The only zeros or poles of $\Psi_f$ lie on the rational quadratic divisors $\lambda^\bot$ for $\lambda\in M$, $\lambda^2 <0$ and are zeros of order $$\sum_{0< x \in\mathbb{R} \atop x\lambda\in M'} c_{x\lambda}(x^2 \lambda^2/2)$$ or poles if this number is negative.
	\end{itemize}
	Translating the above in our setting gives 
	\begin{theorem}{(Borcherds theorem for the Fricke Group)}\label{language} 
		\begin{itemize}
			\item[1.] For some $h$, the Borcherds lift of $f$ is defined as $$ \Psi_f (\tau) = q^{-h} \prod_{n>0} (1-q^n)^{c_\beta(n^2/4p)} $$
			where for each $n>0$ we take the Fourier coefficient having $\beta\equiv n[2p]$.\\ Then $\Psi_f$ is a meromorphic modular form for the group $\Gamma_0(p)^*$ of weight $c_0(0)/2$ and some unitary character $\chi$.
			\item[2.] The possible zeros or poles of $\Psi_f$ occur
			either at cusps or at imaginary quadratic
			irrationals $\alpha\in \HH$ satisfying $pa\alpha^2+b\alpha+c=0$ with $(a,b,c)=1$. Also, letting $\delta=b^2-4pac<0$, we obtain the multiplicity of the zero of $\Psi_f$ at $\alpha$ which is given by $\sum_{n>0} c_\beta(\delta n^2/4p)$
			where
			for each $n>0$,
			$\beta$ is chosen such that
			$\beta^2\equiv\delta n^2  [4p]$.
		\end{itemize}
	\end{theorem}
	We are almost there since the Fourier coefficients $c_\beta$ have the shape we have obtained for the primes $p=2,3,5$. Consider one of the weakly holomorphic modular forms we have constructed previously, say $f_{d,p}$ in $M_{1/2}^!(\Gamma_0(4p))$. It has Fourier expansion of the form
	$q^{-d}+\sum_{D\ge 1} A(D,d)q^D$ and its constant term is $0$. Taking its Borcherds lift, we obtain that $\Psi_{f_{d,p}}$ is a meromorphic modular form of weight $0$ for some unitary character of $\Gamma_0(p)^*$ with $q$-product expansion given by 
	$$ \Psi_{f_{d,p}}(\tau)=q^{-h} \prod_{n=1}^\infty
	(1-q^n)^{A^*(n^2,d)}$$
	for some $h$ where $A^*(n^2,d)=2A(n^2,d)$ if $p|n$ and $A(n^2,d)$ otherwise. Pick an integer $\beta[2p]$ with $\beta^2\equiv -d[4p]$. By the above theorem again, we have the zeros of $\Psi_{f_{d,p}}$ occur at $\alpha_Q$ for every $Q\in Q_{d,p,\beta}$ with multiplicity $2^{s(\beta,p)}=1$ if $$p\not|\beta\iff\beta\not\equiv -\beta[2p]\iff W_p\circ Q\not\in Q_{d,p,\beta}$$ and multiplicity $2^{s(\beta,p)}=2$ if $$p|\beta\iff\beta\equiv -\beta[2p]\iff W_p\circ Q\in Q_{d,p,\beta}$$ Now if we consider the modified Hilbert class polynomial, we can see that the multiplicity of the zero of $j_p(\tau)-j_p(\alpha_Q)$ at $\alpha_Q$ is $w_Q({\Gamma_0(p)^*)}=w_Q({\Gamma_0(p)})$ if $$W_p\circ Q\neq Q[\Gamma_0(p)]$$ and is $w_Q({\Gamma_0(p)^*})=2w_Q({\Gamma_0(p)})$ if $$W_p\circ Q= Q[\Gamma_0(p)]$$ We see that $\Psi_{f_{d,p}}$ and $\mathcal{H}_{d,p}$ have the same divisor on $\HH$. Since some power of $\Psi_{f_{d,p}}$, say $\Psi_{f_{d,p}}^k$, is a modular function for $\Gamma_0(p)^*$ and since $X_0(p)^*$ has only the cusp at $\infty$, the divisor of $\Psi_{f_{d,p}}^k$ on this curve has the form $\sum_{Q\in \Gamma_0(p)\backslash\mathcal{Q}_{d,p,\beta}}k\Big(\frac{1}{w_Q(\Gamma_0(p))}\Big)\alpha_Q-kh\infty$ and degree zero. From this observation we get
	$$\prod_{Q\in \Gamma_0(p)\backslash\mathcal{Q}_{d,p,\beta}}
	(j_p(\tau)-j_p(\alpha_Q))^{\frac{1}{w_Q(\Gamma_0(p))}}=q^{-h}\prod_{n=1}^\infty
	(1-q^n)^{A^*(n^2,d)}$$ where $h=H_p(d)=\sum_{Q\in \Gamma_0(p)\backslash\mathcal{Q}_{d,p,\beta}}\frac{1}{w_Q(\Gamma_0(p))}$.
	\begin{corollary} 
		Let $p|\#\MM$ and $\forall d\equiv \square[4p]$, we define
		$\textbf{\emph{t}}^{(p)}(d)=\sum_{Q\in\mathcal{Q}_{d,p,\beta}/\Gamma_0(p)}
		\frac{1}{|\bar\Gamma_0(p)_Q|} j_p (\alpha_Q)$. We put $\textbf{\emph{t}}^{(p)}(-1)=-1$, $\textbf{\emph{t}}^{(p)}(0)=2$ and $\textbf{\emph{t}}^{(p)}(d)=0$
		for $d<-1$. Then the series $\sum_{n,r} \textbf{\emph{t}}^{(p)}(4pn-r^2) q^n \zeta^r$ is a weakly holomorphic Jacobi form of weight 2 and index $p$ and thus a weakly holomorphic modular form of weight $3/2$ and level $4p$ in the Kohnen plus-space.
	\end{corollary}
	

	
\end{document}